\documentclass[reqno,11pt]{amsart}
\usepackage{graphicx, amsmath, amssymb, amscd, mathtools, hyperref, amsthm, euscript, caption,
amsfonts,bm,bbm,xcolor}  
\DeclareMathAlphabet{\mathpzc}{OT1}{pzc}{m}{it}

\usepackage{wrapfig}
\usepackage{psfrag, graphicx, pgf, float}

\newtheorem{thm}[equation]{Theorem}
\newtheorem{theorem}[equation]{Theorem}
\newtheorem{Ex}[equation]{Example}

\newtheorem{Rmk}[equation]{Remark}\newtheorem{remark}[equation]{Remark}

\newtheorem{prop}[equation]{Proposition}

\newtheorem{cor}[equation]{Corollary}
\newtheorem{lem}[equation]{Lemma}

\newtheorem{Def}[equation]{Definition}

\numberwithin{equation}{section}

\numberwithin{equation}{section}

\newcommand{\be}{begin{equation}}
\newcommand{\bH}{\mathbb H}

\newcommand{\bb}{\mathbb}

\newcommand{\e}{\epsilon}

\renewcommand{\c}{\mathbb{C}}
\newcommand{\br}{\mathbb{R}}

\newcommand{{\grinv}}{{\Cal G}^{-r}}

\newcommand{\ba}{\backslash}

\newcommand{\G}{\Gamma}

\newcommand{\Cal}{\mathcal}

\newcommand{\bp}{\begin{pmatrix}}
\newcommand{\ep}{\end{pmatrix}}
\renewcommand{\be}{\begin{equation}}
\newcommand{\ee}{\end{equation}}

\def\vecv{{\text{\boldmath$v$}}}
\def\vect{{t}}

\def\vecz{{z}}

\def\vecw{{w}}

\renewcommand{\bp}{{\rm bp}}

\newcommand{\CC}{\mathbb C}

\renewcommand{\L}{\Cal L}
\newcommand{\Vol}{\op{Vol}}
\newcommand{\PSL}{\op{PSL}}

\newcommand{\PS}{\op{PS}}
\newcommand{\muPS}{\mu^{\PS}}

\newcommand{\norm}[1]{\lVert #1 \rVert}

\newcommand{\op}{\operatorname}

\newcommand{\BR}{\operatorname{BR}}

\newcommand{\cl}[1]{\overline{#1}}
\newcommand{\pu}{\Phi_\rho(u)}
\renewcommand{\muPS}{\mu^{\PS}}
\renewcommand{\setminus}{-}

\newcommand{\Ad}{\operatorname{Ad}}

\newcommand{\Ga}{\Gamma}

\def\RR{\mathbb{R}}

\def\om{\omega}

\newcommand{\ga}{\gamma}
\newcommand{\F}{\mathcal F}
\newcommand{\La}{\Lambda}
\renewcommand{\i}{\op{i}}

\def\scrO{{\mathcal O}}
\def\scrP{{\mathcal P}}

\def\e{\mathrm{e}}
\def\i{\mathrm{i}}

\def\PSL{\operatorname{PSL}}

\newcommand{\smatr}[4]{\left( \begin{smallmatrix} #1 & #2 \\ #3 & #4 \end{smallmatrix} \right) }

\newcommand{\Zar}{{{Zariski-dense }}}
\newcommand{\Conv}{{convex-cocompact }}

\newcommand{\GaG}{\Gamma_\rho\backslash G}

\newcommand{\sfrac}[2]{{\textstyle \frac {#1}{#2}}}

\newcommand{\fa}{\mathfrak{a}}

\newcommand{\cal}{\mathcal}
\renewcommand{\e}{\varepsilon}
\renewcommand{\epsilon}{\e}

\newcommand{\inte}{\op{int}}

\newcommand{\mBR}{m_\psi^{\mathrm{BR}}}

\newcommand{\Gr}{\Gamma_\rho}
\begin{document}

\title[Torus counting]{Torus counting and self-joinings of Kleinian groups}

\author{Sam Edwards, Minju Lee, and Hee Oh}
\address{Department of Mathematical Sciences, Durham University, Lower Mountjoy, DH1 3LE Durham, United Kingdom}
\address{Mathematics department, University of Chicago, Chicago, IL 60637, USA}
\address{Mathematics department, Yale university, New Haven, CT 06520, USA and Korea Institute for Advanced Study, Seoul, Korea}
\address{}
\thanks{Edwards's work was supported by the Additional Funding Programme for Mathematical Sciences, delivered by EPSRC (EP/V521917/1) and the Heilbronn Institute for Mathematical Research. Oh is partially supported by the NSF grant No. DMS-1900101.}
\begin{abstract}
For any $d\ge 1$, we obtain counting and equidistribution results for tori with small volume for a class of $d$-dimensional torus packings, invariant under a self-joining $\Ga_\rho<\prod_{i=1}^d\PSL_2(\c)$
of a Kleinian group $\Ga$
formed by a $d$-tuple of \Conv representations $\rho=(\rho_1, \cdots, \rho_d)$. 
More precisely, if $\cal P$ is a $\Ga_\rho$-admissible $d$-dimensional torus packing, then
for any bounded subset $E\subset \c^d$  with $\partial E$ contained in a proper real algebraic subvariety, we have
$$\lim_{s\to 0} { s^{\delta_{L^1}({\rho}) }} \cdot \#\{T\in \cal P: \Vol (T)> s,\,  T\cap E\ne \emptyset  \}=
c_{\cal P}\cdot \omega_{\rho} (E\cap \Lambda_\rho).$$
Here $0<\delta_{L^1}(\rho)\le  2/\sqrt d$ is the critical exponent of $\Ga_\rho$ with respect to the $L^1$-metric on the product $\prod_{i=1}^d \bH^3$, $\Lambda_\rho\subset (\c\cup\{\infty\})^d$ is the limit set of $\Ga_\rho$, and 
$\om_{\rho}$
is a locally finite Borel measure  on $\c^d\cap \Lambda_\rho$
which can be explicitly described.
The class of admissible torus packings we consider arises naturally from the Teichm\"{u}ller theory of Kleinian groups. Our work extends previous results of Oh-Shah \cite{OS} on circle packings (i.e.\ one-dimensional torus packings) to $d$-torus packings.
\end{abstract}

\maketitle
\tableofcontents
\section{Introduction}\label{sec.int}

In this paper, we obtain counting and equidistribution results for a certain class of $d$-dimensional torus packings invariant under self-joinings of Kleinian groups for any $d\ge 1$. One-dimensional torus packings are precisely circle packings. To motivate the formulation of our main results, we begin by reviewing counting results for circle packings that are invariant under Kleinian groups (\cite{KO}, \cite{Oh}, \cite{OS}, \cite{OS2}, \cite{PP}, etc).

\subsection*{Circle counting} 
A circle packing  in the complex plane $\c$ is simply a 
 nonempty family of circles in $\c$, for which we allow intersections among themselves. 
In the whole paper, lines are also considered as circles of infinite radii.
Let $\G<\PSL_2(\c)=\op{Isom}^+(\bH^3)$ be a 
\Zar\Conv discrete subgroup.
We call  a circle packing $\cal P$ $\G$-admissible if
\begin{itemize}
    \item $\cal P$ consists of finitely many $\Ga$-orbits of circles;
    \item $\cal P$ is locally finite, in the sense that no infinite sequence of circles in $\cal P$ converges to a circle.
\end{itemize}
 We denote by $0<\delta_\Ga\le 2$  the critical exponent of $\Ga$ i.e.\ the abscissa of convergence for the Poincare series $\mathsf P(s):=\sum_{g\in \Ga} e^{-s \, \mathsf{d}_{\bH^3}(g{p}, {p})}$ where ${p}\in \bH^3$ is any point and
$\mathsf d_{\bH^3}$ is the hyperbolic metric so that $(\bH^3, \mathsf d_{\bH^3})$ has constant curvature $-1$. 
 The extended complex plane $\hat \c=\c\cup\{\infty\}$ can be regarded as the geometric boundary of $\bH^3$. The limit set of $\Ga$ is the set of all accumulation points of {the orbit $\Ga (z)$ of} $z\in \hat \c$; we denote it by $\La_\Ga\subset \hat \c$.
\begin{theorem}\cite{OS} \label{cc} For any $\G$-admissible circle packing $\mathcal{P}$, there exists a constant $c_{\mathcal{P}}>0$ such that for any bounded {measurable} subset $E\subset \c$ whose boundary is contained in a proper real
algebraic subvariety of $\c$,

$$
\lim_{s\to 0}  s^{\delta_\Ga} \#\{C\in \cal P: \op{radius}( C) \ge s,\;  C\cap E\ne \emptyset \} = c_{\cal P}  \; \omega_\Ga (E\cap \La_\Ga); 
$$
here  $\omega_\Ga$ is the $\delta_{\Ga}$-dimensional Hausdorff measure on $\mathbb C\cap \La_\Gamma $ with respect to the Euclidean metric on $\mathbb C$.
 \end{theorem}
 
This theorem holds for a more general class of circle packings invariant by geometrically finite Kleinian groups, which includes the famous Apollonian circle packings for which the relevant counting result was first obtained in \cite{KO} (see \cite{OS} for more details and examples).

\subsection*{Torus counting} The main goal of this paper is
to prove  a higher dimensional analogue of Theorem \ref{cc}. Let $d\ge 1$.
{By a} torus in $\c^d$   
{ we mean }
a Cartesian product of $d$-number of circles $C_1,\cdots, C_d\subset \c$.
{
However, it will be convenient to consider it as a $d$-tuple of circles}
\begin{equation}\label{eq.DFT}
T= 
(C_1, \cdots, C_d)        
\end{equation}
{
rather than a subset $C_1\times \cdots \times C_d\subset\bb C^d$.
}
A $d$-dimensional torus packing in $\c^d$ is simply a  { nonempty family} of $d$-tori in $\c^d$.

The volume of $T$ is given by
$$\Vol (T)= \prod_{i=1}^d 2\pi  {\;} \text{radius } C_i.$$

Figure 1 shows some image of a $2$-torus packing. Although the torus $T=C_1\times C_2$ in Fig.\ 1 {\it appears} to be in $\br^3$, it should be understood as a subset of $\br^4$, representing the Cartesian product of the boundary circles of two discs.
\begin{figure}[h]
\centering \captionsetup{justification=centering}
      \includegraphics[totalheight=5cm]{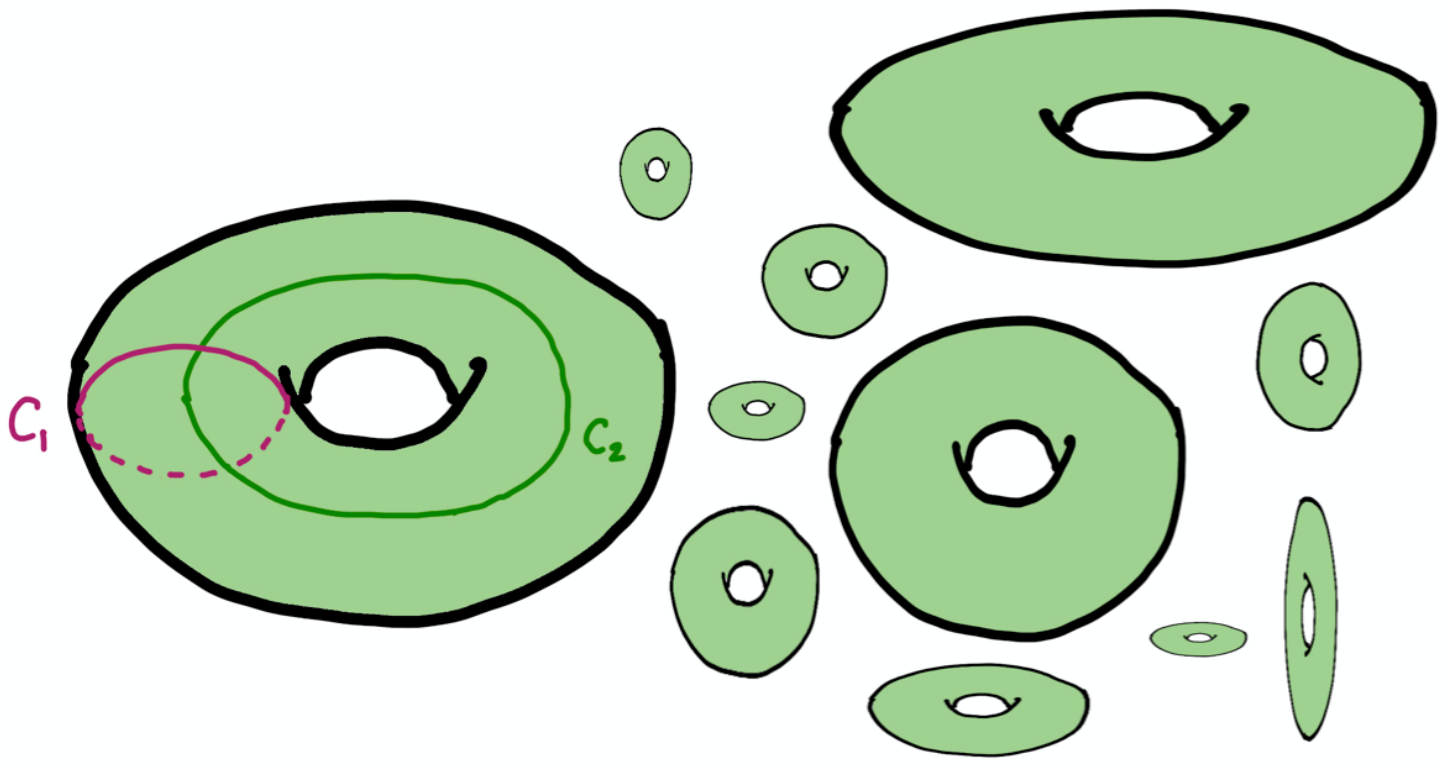}
\caption{A torus packing}
\label{TP}
\end{figure}

We are interested in understanding the asymptotic counting and distribution of tori with small volumes in a torus packing that is invariant under a self-joining of a \Conv Kleinian group.

Let $\Ga<\PSL_2(\c)$ be a \Conv discrete subgroup and
$\rho=(\rho_1=\text{id}, \rho_2, \cdots, \rho_d)$ be a $d$-tuple of faithful \Conv representations of $\Ga$ into $\PSL_2(\c)$. Let $G=\prod_{i=1}^d \PSL_2(\c)$.
The { {\it self-joining}} of $\Gamma$ via $\rho$ is defined as the following discrete subgroup of $G$:
$$\Gamma_\rho 
=\{\big(\rho_1(g), \cdots , \rho_d(g)\big): g\in \Gamma\}.$$
Throughout the paper we will always assume that $\Ga_\rho$ is \Zar in $G$.
Each $\rho_i$ induces a unique equivariant homeomorphism  $f_i:\La_{\Ga}\to \La_{\rho_i(\Ga)}$, which is called the $\rho_i$-boundary map \cite{Tu}.
{In this paper, we define }
the limit set of $\Ga_\rho$  { by} 
$$\La_\rho=\{(f_1(\xi), \cdots, f_d(\xi))\in \hat \c ^d:\xi\in \La_\Ga\}.$$

 We call a torus $T=(C_1, \cdots, C_d)$ $\Ga_\rho$-admissible
if for each $1\le i\le d$,
\begin{itemize}
    \item
$\rho_i(\Gamma) C_i$ is a locally finite circle packing;
\item $f_i(C_1\cap \La_\Ga) =C_i \cap \La_{\rho_i(\Ga)} .$
\end{itemize}

The second condition is equivalent to$$T\cap \La_\rho=\{(\xi_1,\cdots, \xi_d)\in \La_\rho: \xi_1\in C_1\cap \La_\Ga\},$$ that is,  the circular slice $C_1\cap \La_\Ga$ completely determines
the toric slice $T\cap \La_\rho$.

\begin{Def}\label{in-ad}  A torus packing $\cal P$ is called $\Ga_\rho$-admissible if
\begin{itemize}
    \item 
$\cal P$ consists of finitely many $\Ga_\rho$-orbits of $\Ga_\rho$-admissible tori;\item 
$\cal P$ is locally finite in the sense that no infinite sequence of tori in $\cal P$ converges to a torus.
\end{itemize}
\end{Def}

\begin{Rmk} \rm We remark that when $\#(C_1\cap \La_\Ga) \ge 3$,
the locally finiteness hypotheses in the above definition can be reduced to the local-finiteness of the circle packing $\Ga C_1$ (see Prop. \ref{ccc}).
\end{Rmk}

We denote by $\delta_{L^1}({\rho})$ 
the abscissa of convergence of the series 
$$
s\mapsto \mathsf P_{L^1}(s):=\sum_{g\in \Ga} e^{-s \sum_{i=1}^d \mathsf d_{\bH^3}(\rho_i(g)  {p}, {p})}
$$ for ${p}\in \bH^3$, which is the critical exponent of $\Ga_\rho$ with respect to the $L^1$ product metric
on $\prod_{i=1}^d (\bH^3, \mathsf d_{\bH^3})$.

We first state the following special case of the main result of this paper{.} \begin{thm}\label{m1}\label{THM2}  Let $\cal P$ be a $\Ga_\rho$-admissible torus packing.
There exists a constant $c_{\cal P}>0$  such that for any bounded {measurable} subset $E\subset \c^d$ with boundary contained in a proper real algebraic subvariety, we have
$$\lim_{s\to 0} { s^{\delta_{L^1}({\rho}) }} {\;} \#\{T\in \cal P: \Vol (T)> s,\,  T\cap E\ne \emptyset  \}=
c_{\cal P}{\;} \omega_{\Ga_\rho} (E\cap \Lambda_\rho),$$
where $\om_{\Ga_\rho}$ is a locally finite Borel measure  on $\c^d\cap \Lambda_\rho$ which can be explicitly described. In particular, if $\cal P$ is bounded, then $$\lim_{s\to 0} s^{\delta_{L^1} (\rho)}{\;}
{\#\{T\in \cal P:  \Vol(T) >s\}}=c_{\cal P}{\;|\om_{\Ga_\rho}|}  .$$ \end{thm}

\begin{Rmk} \rm \begin{enumerate}
    \item 
Since $\delta_{L^1}(\rho)$ is bounded above by the
usual critical exponent $\delta_{\Ga_\rho}$ of $\Ga_\rho$ with respect to the Riemannian metric (which equals the $L^2$ product metric) on $\prod_{i=1}^d \bH^3$, we have 
$$0<\delta_{L^1(\rho)}\le  \delta_{\Ga_\rho} \le  \frac{1}{\sqrt d}\max_i (\text{dim}(\La_{\rho_i(\Ga)})) \le  \frac{2}{\sqrt d} $$
by
\cite[Coro. 3.6]{KMO}; here the notation $\text{dim}(\cdot)$ means the Hausdorff dimension of a {measurable} 
 subset of $\hat \c\simeq \mathbb S^2$ with respect to the spherical metric.
\item If all $\rho_i:\Ga\to \PSL_2(\c)$ are quasiconformal deformations of $\Ga$ and $\infty\notin \cup_{i=1}^d \La_{\rho_i(\Ga)}$, then for any
 bounded torus packing $\cal P=\Ga_\rho T$ with $T=(C_1, \cdots, C_d)$, 
$\cal P$ is locally finite if and only if { $\{\rho_i(\ga)C_i:\ga\in\Ga\}$} is a locally finite circle packing for all $1\le i \le d$. This is because the boundary map $f_i$ is the restriction to $\La_{\rho_i(\Ga)}$ of the quasiconformal homeomorphism $F_i:\hat\c\to \hat \c$ associated to $\rho_i$, and under the hypothesis $\infty\notin \cup_{i=1}^d \La_{\rho_i(\Ga)}$, the $F_i$ are bi-H\"older
maps on any compact subset of $\c$ (\cite{FS}, \cite{Tu}).

\end{enumerate}

\end{Rmk}

\subsection*{More general torus-counting theorems} In order to present a more general torus-counting theorem, we define the length vector of a torus $T=(C_1,\cdots,C_d)$ by
$$\vecv(T)=-\big(\log \text{radius}(C_1), \cdots, \log \text{radius}(C_d) \big)\in\br^d ;$$
where we used the negative sign so that the $i$-th coordinate of
$\vecv (T)$ tends to $+\infty$ as $C_i$ shrinks to a point.
 The following { result} is the main theorem of this paper{.}
\begin{thm}\label{thm.main.intro}  Let $\psi$ be any linear form on $\br^d$ such that $\psi>0$ on ${ (\br_{\ge 0})}^d-\{0\}$. There exist $\delta_\psi >0$ and a locally finite Borel measure $\omega_\psi$ on $\La_\rho\cap {\mathbb C}^d$ depending only on $\Ga_\rho$ and $\psi$ for which the following hold:
for any $\Ga_\rho$-admissible torus packing $\cal P$,
there exists a constant $c_\psi=c_{\cal P,\psi}>0$ such that for any bounded {measurable} subset $E\subset \c^d$ with boundary contained in a proper real algebraic subvariety, we have, as $R\to \infty$,
\begin{equation}\label{eq.E}
\lim_{R\to \infty} \frac{1}{ e^{\delta_\psi R}} {\#\{T\in \cal P: \psi (\vecv(T)) <R, \, T\cap E\ne \emptyset\} } =
c_{\psi}{\;} \omega_{\psi} (E\cap \Lambda_\rho).
\end{equation}
\end{thm}

The description of the measure $\omega_\psi$ (Def. \ref{eq.om})
depends on the higher rank Patterson-Sullivan theory. In fact, it is  equivalent to the unique $(\Ga_\rho,\psi_0)$-conformal measure on $\La_\rho$, where $\psi_0$ is { the } unique $\Gamma_\rho$-critical linear form (Def. \ref{def.crit})
proportional to $\psi$.
We refer to Def. \ref{dp} for the definition of $\delta_\psi$.

\begin{remark}\label{Euc} \rm
\begin{enumerate}
\item Theorem \ref{m1} can be deduced from this theorem by considering the linear form ${\psi} : (t_1, \cdots, t_d)\mapsto t_1+\cdots+t_d$ (see Ex. \ref{tr}). 
\item Our approach can also handle the case where $\psi(\vecv (T))$ is replaced by the Euclidean norm of $\vecv(T)$ in \eqref{eq.E}; indeed, the analysis involved in that case is easier due to the strict convexity of the Euclidean balls in $\br^d$ (see the last subsection of Sec \ref{pm}).
\item The fact that the sublevel sets $\{t\in \br^d: \psi(t) < c\}$ are linear (hence not strictly convex) presents new technical difficulties which were not dealt with in related previous works such as \cite{OS} and \cite{ELO}.

\end{enumerate}
\end{remark}

We now discuss examples of admissible torus packings arising naturally from the Teichm\"{u}ller theory of Kleinian groups. 

\begin{figure}[h]
\centering \captionsetup{justification=centering}
      \includegraphics[totalheight=4cm]{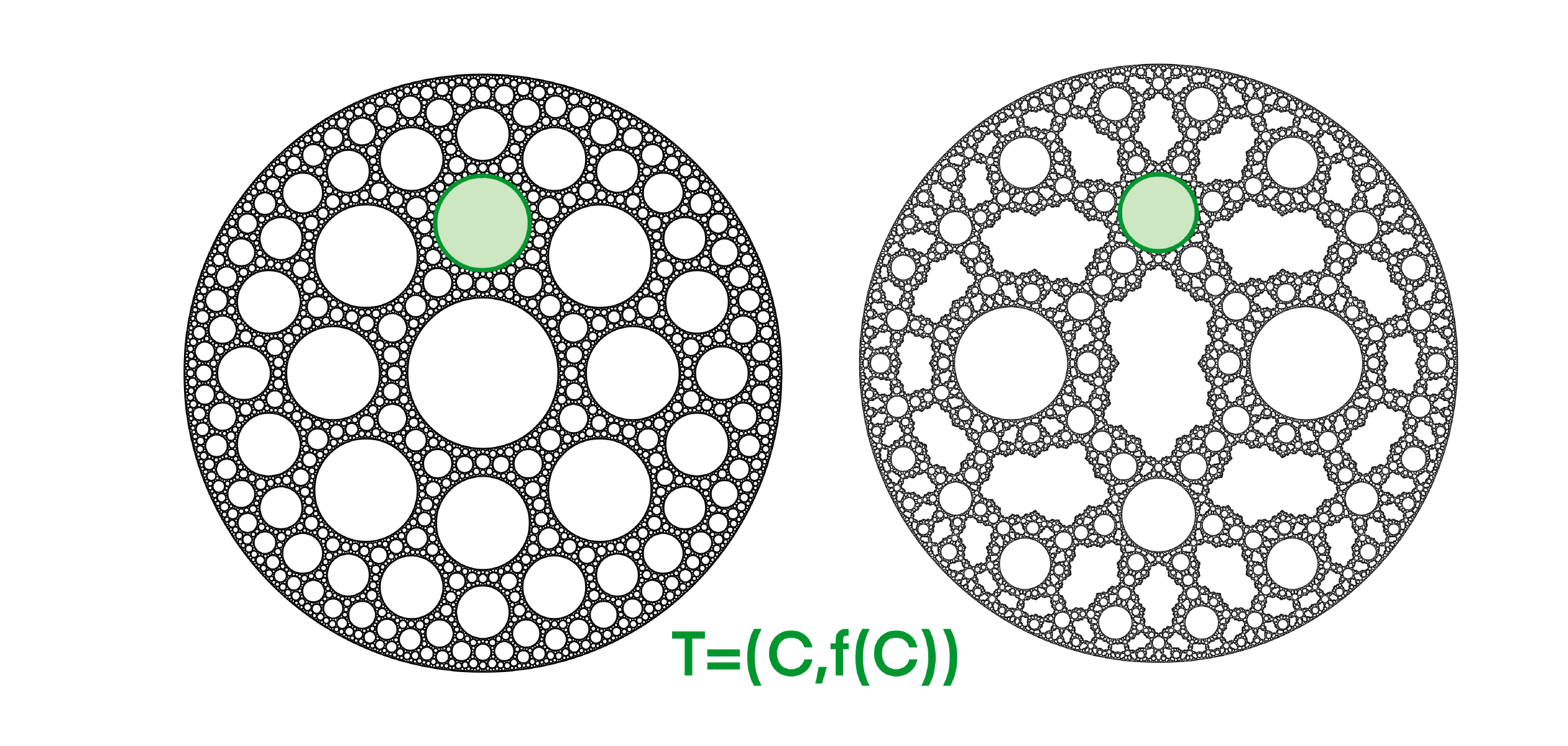}
\caption{{The left-hand side is the limit set of a \Conv Kleinian group $\Ga$
and the right-hand side is the limit set of a quasi-conformal deformation, say, $\rho_0$, of $\Ga$.
 Denoting by $f$ the associated quasiconformal map, $f$ maps the first green circle, say $C$, to the second green circle. Hence the torus  $T=(C, f(C))$ is a $(\text{id}\times \rho_0)(\Gamma)$-admissible torus.}
 (image credit: Yongquan Zhang)}
\label{couple}
\end{figure}  

\begin{Ex}\rm \begin{enumerate}
    \item 
 Let $\Ga<\PSL_2(\c)$ be a Zariski-dense and convex-cocompact subgroup whose domain of discontinuity $\Omega_\Ga:=\hat \c -\La_\Ga$ has a connected component which
is a round open disk $B$.  Let $C_1:=\partial B$ and $d\ge 2$. By the Teichm\"{u}ller theory of $\Gamma$, which relates the Teichm\"{u}ller space of the Riemann surface $\Ga\ba \Omega_\Ga$ and the quasi-conformal deformation space of $\Ga$ (\cite[Thm. 5.27]{MT}, \cite{Mar}) we may choose quasi-conformal deformations
$\rho_i:\Ga\to \PSL_2(\c)$, $2\le i\le d$, whose associated quasiconformal maps 
$f_i:\hat \c\to \hat \c$ map $C_1$ to a circle, say, $C_i$. Then
$T=(C_1, \cdots, C_d)$ is a $\Ga_\rho$-admissible torus for $\rho=(\text{id}, \rho_2, \cdots, \rho_d)$  and hence $\cal P=\Ga_\rho T$ is a $\Ga_\rho$-admissible torus packing (see Figure \ref{couple} for an example when $d=2$). Note also that $\cal P$ consists of {\it disjoint} tori, and hence gives rise to a genuine packing. 

\item 
Let $\Ga$ be 
{ a rigid acylindrical \Conv Kleinian group}, that is, $\Omega_\Ga$ is a union of infinitely many round disks with mutually disjoint closures.
Let $\rho_0:\Ga\to \PSL_2(\c)$ be a quasiconformal deformation of $\Ga$ which is not a conjugation, and $f:\hat\c\to \hat \c$ the associated quasiconformal map.
{
Denoting by $\cal C$ the space of all round circles in $\hat{\bb C}$,}
it follows from  (\cite{MMO}, \cite{MMO2}, \cite{BO})  that the set of all circles 
{ $C\in\cal C$} such that $\# C\cap \La_\Ga\ge 2$ and $f(C)$
is a circle is a finite union of closed $\Ga$-orbits in $\cal C$. Indeed, if $C\in \cal C$ meets $\La_\Ga$ at more than one point, then either $C$ separates $\La_\Ga$ or $C\subset \La_\Ga$.
Since the set of circles contained in $\La_\Ga$ is a finite union of closed $\Ga$-orbits, it suffices to note that the set of all separating circles such that $f(C)$
is a circle is a finite union of closed $\Ga$-orbits. This follows from \cite[Thm. 1.5]{MMO} and \cite[Thm. 1.6]{BO}, since
otherwise such a set must be dense in  the space $\cal C_{\La_\Ga}$ of all circles
meeting $\La_\Ga$, and hence $f$ must map all circles in $\cal C_{\La_\Ga}$ to circles. That implies that $f$ is conformal \cite{MT} and hence $\rho$ is a conjugation, a contradiction.

Therefore
the following $2$-dimensional torus packing
$$\cal P:=\{(C, f(C)) :C, f(C) \text{ are circles and  } \; \# C\cap \La_\Ga\ge 2\} $$
is  $(\op{id}\times \rho_0)(\Ga)$-admissible.
\end{enumerate}
  \end{Ex}

\subsection*{On the proof of Theorem \ref{thm.main.intro} }
First of all, the self-joining group
$\Ga_\rho$ is { an} \textit{Anosov subgroup} of $G$ introduced in \cite{GW} (see Def. \eqref{eq.Q}), which enables us to apply the general ergodic theory developed for Anosov subgroups. While certain types of counting problems for orbits of Anosov subgroups in affine symmetric spaces were studied in our earlier paper \cite{ELO} using higher rank Patterson-Sullivan theory, there were certain serious technical restrictions imposed in
\cite{ELO} which made it unclear what kind of torus packing counting problems could be approached using techniques there.
 One of the main novelties of this paper is to have isolated a natural class of torus packings (which are provided by the Teichm\"{u}ller theory of Kleinian groups) for which we can apply the counting machinery of \cite{ELO}.

It is not hard to reduce  the proof of Theorem \ref{thm.main.intro} to the case where
 $\cal P$ is of the form $\Ga_\rho T_0$, where $T_0$ is the product of
the unit circles centered at the origin and $\psi$ is a so-called $\Ga_\rho$-critical linear form (see Def. \ref{def.crit}). As in \cite{OS}, we first translate the counting problem for torus packings into an orbital counting problem in $H\ba G$ where $H=\op{Stab}_G(T_0)$; by introducing a suitable bounded {measurable} 
 subset $B_\psi(E,R)\subset H\ba G$ in \eqref{BER}, we are led to consider the asymptotic of 
\begin{equation*}\label{eq.G}
\#\big([e]\Ga_\rho\cap  {B}_\psi(E,R)\big)    
\end{equation*}
as $R\rightarrow\infty$. The key ingredient for obtaining \eqref{eq.E} as $R\to\infty$ is 
a description of the asymptotic behavior of 
\be \label{inner} \int_{{B_\psi}(E,R)} \left(\int_{\Ga_\rho\cap H \ba H} f( [h] g)\,d[h]\right) \,d[g]\ee 
for $f\in C_c(\Ga_\rho\ba G)$, as $R$ tends to infinity,
as given in Theorem \ref{thm3}.
The $\Ga_\rho$-admissibility assumption on $\cal P=\Ga_\rho T_0$ is used
to guarantee 
\begin{itemize}
    \item 
the existence of some compact subset $S\subset \Ga_\rho\cap H\ba H$, {\it independent of $R$}, such that
the integral \eqref{inner} can be expressed as
\be \label{inner2} \int_{[g]\in {B_\psi}(E,R)} \left(\int_{[h]\in S} f( [h] g)\,d[h]\right) d[g]\ee 

\item the finiteness of the skinning constant of $\Ga_\rho\cap  H\ba H$ (see \eqref{eq.sk}).
\end{itemize}

With this information, as well as the analysis of the asymptotic shape of the family of the subsets
 $\{B_\psi(E,R): R>0\}$, we are able to apply the mixing result from \cite[Thm. 3.4]{ELO2} and \cite[Thm. 1.3 \& Thm. 1.4]{CS}, and the equidistribution result from \cite{ELO} which describes
 the asymptotic of the integral \eqref{inner} in terms of the Burger-Roblin measures introduced in \cite{ELO}.
We emphasize that due to the higher rank nature of the subsets $B_\psi(E,R)$, combined with the
linear nature of $\psi$, whose sublevel sets are not strictly concave, the uniformity aspect
in these results (see Propositions \ref{prop.mixH} and \ref{prop.mixH2} for the nature of the uniformity that is required) is crucial for our analysis.
In fact, working on this article led us to conjecture the precise uniformity formulation of the mixing results in \cite{CS}, which were verified and appeared in an updated version by the authors. Finally, we remark that the measure $\omega_\psi$
is the leafwise measure of the Burger-Roblin measure on the strict upper triangular subgroup of $G$ ($\simeq \c^d$) (see Proposition \ref{BRPS}).

\subsection*{Organization}
\begin{itemize}
    \item 
    In Section \ref{sec:pre}, we start by recalling the basic higher rank Patterson-Sullivan theory of self-joining groups.

    \item
    In Section \ref{sec:adm}, we discuss an important property of $\Ga_\rho$-admissible torus packings and its consequences.
    \item 
    In Section \ref{sec:count}, we define the family $\{B_\psi(E,R)\subset H\ba G: R>0\}$ and explain how Theorem \ref{THM2} can be translated into an orbital-counting
    problem for a $\Ga_\rho$-orbit in $H\ba G$ with respect to the family $\{B_\psi(E,R): R>0\}$.
   
    \item
    In Section \ref{uSec}, mixing and and equidistribution results  from \cite{CS} \cite{ELO} will be recalled with an emphasis on their uniformity aspects.
  
    \item In Section \ref{omegaSec}, the measure $\om_{\psi}$ will be given explicitly and analyzed.
      \item
    In Section \ref{sec.main}, we prove the key technical ingredient (Theorem \ref{thm3}) of the paper, which accounts for the asymptotic distribution of the average of translates of the $H$-orbit over the set $B_\psi(E,R)$ as $R\to\infty$.
    \item 
    In Section \ref{pm}, we prove the main theorem (Theorem \ref{THM2}).

    \item 
    In Section \ref{sec:f}, we prove that every proper subvariety of $\bb C^d$ has zero Patterson-Sullivan measure and hence zero $\om_\psi$ measure; this is shown for a general Anosov subgroup of a semisimple real algebraic group. 
  
\end{itemize}

\subsection*{Acknowledgements}
We would like to thank Dongryul Kim for useful conversations on a related topic.
\section{Self-joinings  and higher rank Patterson-Sullivan theory}\label{sec:pre}
Let $\bH^3=\{(z, r):z\in \c, r>0\}$ denote the upper {halfspace} model
of hyperbolic $3$-space with constant curvature $-1$, $\mathsf d
$ the hyperbolic metric on $\bH^3$ and {$o=(0,1)\in\bH^3$}.
The geometric boundary of $\bH^3$ is the extended complex plane $\hat{\bb C}:=\bb C\cup\{\infty\}$, which is the Riemann sphere.
The M\"obius transformation action of 
the group $\op{PSL}_2(\bb C)$ on $\hat \c$ 
extends to the action on the compactification $\bH^3\cup \hat \c$, and gives rise to the identification  $\PSL_2(\c)\simeq \op{Isom}^\circ(\bH^3)$, the identity component of the isometry group of $\bH^3$. 
Similarly, the product group
$$
G=\prod_{i=1}^d \PSL_2(\c)
$$
acts on $\hat{\bb C}^d$ component-wise, giving rise to an isomorphism of
 the group $G$ with $\op{Isom}^\circ(\prod_{i=1}^d \bb H^3)$, the identity component of the isometry group of the Riemannian product 
 { $(\bH^3)^d$}.

\subsection*{Self-joinings of \Conv subgroups}
Let $\Ga<\op{PSL}_2(\bb C)$ be a torsion-free \Conv subgroup, that is, the convex core of the associated hyperbolic manifold $\Ga\ba \bH^3$ is compact.

Let $\rho=(\rho_1=\text{id}, \rho_2, \cdots, \rho_d)$ be a $d$-tuple of faithful \Conv representations of $\Ga$ into $\PSL_2(\c)$, i.e.\ each $\rho_i(\Ga)$ is a \Conv subgroup of $\PSL_2(\c)$.
 \begin{Def}\rm  The self-joining of $\Ga$ by $\rho$ is defined as the following discrete subgroup of $G$:
$$
\Gamma_\rho 
=\{(\rho_1(g), \cdots  \rho_d(g))\in G: g\in \Gamma\}.
$$
\end{Def} 

Recall that throughout the entire paper we assume that
$$\text{$\Ga_\rho$ is \Zar in $G$.}$$
\subsection*{Anosov subgroups}
Let $|\cdot|$ denote the word length on $\Ga$ with respect to a fixed finite { generating set}.
Since each $\rho_i$ is \Conv, there exists $C>0$ such that
\begin{equation}\label{eq.Q}
\mathsf{d}(\rho_i(g)o,o)>C|g|-C^{-1}\text{ for all }g\in\Ga\text{ and }1\leq i\leq d.
\end{equation}
In other words,
$\Ga_\rho$ is an \textit{Anosov subgroup} (with respect to a minimal parabolic subgroup) ({see }\cite{KLP2} and \cite{GW}).
 This is the most important feature of the self-joining $\Ga_\rho$ which will be used in this paper.
We remark that any Anosov subgroup of $G$ arises in this way {in view of the characterization \cite[Thm 1.5]{KLP2}}.

\subsection*{Limit set} The product $\cal F=\hat \c^d$ is equal to the Furstenberg boundary of $G$; note that for $d>1$, $\cal F$ is not the geometric boundary of $\prod_{i=1}^d\bH^3$.
Let $P<G$ be the product of the upper triangular subgroups of the $\PSL_2(\c)$ components of $G$, i.e.,
$P=\op{Stab}_G(\infty, \cdots,\infty)$. Then
$$\cal F\simeq G/P.$$

The \textit{limit set} of $\Ga_\rho$ { in $\cal F$} is defined as the {set} of all accumulation points of {any} 
$\Ga_\rho$-orbits
in $\prod_{i=1}^d\bH^3$
on  $\cal F=\hat \c^d$:
\begin{equation*}\label{eq.lim}
{\La_{\rho}}:=\{\lim_{j\to\infty} (\rho_1(g_j)o, \cdots,
\rho_d(g_j)o)\in \hat{\bb C}^d : g_j\in\Ga,\, g_j\rightarrow\infty\}.
\end{equation*}
This definition coincides with the definition of the limit set given by Benoist (\cite[Lemma 2.13]{LO}, \cite{B}).
Note that for $d=1$, this is the usual limit set { $\La_\Ga$ of the Kleinian group $\Ga$}.
{ Let} $\La_{\rho_i(\Ga)}\subset\hat{\bb C}$ denote
the {usual} limit set of $\rho_i(\Ga)$. 

By the \Conv assumption on $\rho_i$, there  exists 
a unique $\rho_i$-equivariant homeomorphism $f_i : \La_{\Ga}\to \La_{\rho_i(\Ga)}$:
\begin{align}\label{eq.EQV}
    f_i(g\xi)=\rho_i(g)f_i(\xi)\quad \text{ for all }g\in\Ga\text{ and }\xi\in\La_\Ga.
\end{align}
In particular, we have
$$\La_\rho=\{ (f_1(\xi), \cdots, f_d(\xi)):\xi\in \La_\Ga\} .$$

\subsection*{Cartan projection} For $t=(t_1, \cdots, t_d)\in\bb R^d$, set
\be\label{at}
a_t=\left(\smatr{e^{t_1/2}}{0}{0}{e^{-t_1/2}}, \cdots, \smatr{e^{t_d/2}}{0}{0}{e^{-t_d/2}}\right).
\ee
We let
$$
A=\{a_t:t\in\bb R^d\}<G\;\; \text{ and } \;\;
A^+=\{a_t: t_i\ge 0 \text{ for all $1\le i\le d$}\}.
$$
We respectively identify $\bb R^d$  and  $\bb R_{\geq 0}^d$ with the Lie algebra $\frak a=\log A$ and its positive Weyl chamber $\frak a^+=\log A^+$  via the map $t\mapsto\log a_t$.
For $g=(g_1, \cdots, g_d)\in G$, the \textit{Cartan projection} of $g$ is defined as 
 $$\mu(g)=(\mathsf d(g_1o, o), \cdots, \mathsf d(g_do, o))\in \frak a^+.$$

\subsection*{Limit cone and its dual cone}
\begin{Def}\label{lc}\rm 
The \textit{limit cone} of $\Ga_\rho$ is
 the asymptotic cone of $\{\mu(\ga)\in { (\br_{\ge 0})}^d:\ga\in\Ga_\rho\}$, which we denote by $\cal L_\rho$.
Alternatively, it is the smallest closed cone in $\frak a^+$
containing $\{\big(\ell_1(g), \cdots, \ell_d(g) \big): g\in\Ga\}$, 
where $\ell_i(g)$ denotes the length of the closed geodesic representing the conjugacy class of $\rho_i(g)$ (\cite{Bu}, \cite[Thm. 1.2]{B}).
\end{Def}

Since  $\sup_{g\in\Gamma} (\ell_i(g)/\ell_j(g))<\infty$  for all $i,j$  by the { convex-cocompactness } assumption, we have 
$$\cal L_\rho \setminus\lbrace 0\rbrace\subset \op{int}\frak a^+,$$
where $\op{int}\cal C 
$ denotes the interior of a cone $\cal C$.
We denote by $\fa^*$ the space of all linear forms on $\fa$. The dual cone of $\cal L_\rho$ is given by
$$
\cal L_\rho^*:=\{\psi\in\frak a^*: \psi|_{\cal L_\rho}\geq 0\}.
$$

Note that 
$$\psi|_{\cal L_\rho-\{0\}}>0\text{  if and only if } \psi\in\op{int}\cal L_\rho^* .$$

\begin{Def} \label{dp} \rm For $\psi\in\op{int}\cal L_\rho^*$,
let $\delta_\psi{\in [0,\infty]}$ denote the abscissa of convergence for the series
$$
s\mapsto \sum_{\ga\in\Ga_\rho}e^{-s\psi(\mu(\ga))}.
$$
\end{Def} 

\subsection*{Critical linear forms}
Let $\norm{\cdot}$ denote the Euclidean norm on $\frak a=\br^d$.
The \textit{growth indicator function} $\Phi_{\rho} : \frak a^+\to\bb R
\cup\{-\infty\}
$ \cite[{$\mathsection$4.2}]{Qu1} is defined as follows:
{ $\Phi_\rho(0)=0$ and} 
for any vector $u\in \frak a^+{-\{0\}}$,
\begin{equation}\label{grow}
\Phi_{\rho}(u):=\|u\| \inf_{\underset{u \in\cal D}{\mathrm{open\;cones\;}\cal D\subset \frak a^+}}
\tau_{\cal D},
\end{equation}
where $\tau_{\cal D}$ is the abscissa of convergence of the series 
$$
\mathsf P_{\cal D}(s)=\sum_{\ga\in\Ga_\rho,\,\mu(\ga)\in\cal D} e^{-s\norm{\mu(\ga)}}.
$$
\begin{Def}\label{def.crit}
A linear form $\psi\in\frak a^*$ is said to be $\Ga_\rho$-\textit{critical} if \begin{itemize}
    \item $\psi\geq\Phi_\rho$ on $\frak a^+$;
\item $\psi(u)=\Phi_\rho(u)$ for some  $u\in\frak a^+-\{0\}$.
\end{itemize}
\end{Def}

The following lemma is due to Quint{.}
\begin{lem}\cite[{Thm. 4.2.2, Lem. 3.1.3, 3.1.7}]{Qu1}\label{lem.H}
\begin{itemize}
    \item For each $\psi\in\op{int}\cal L_\rho^*$, there exists $s>0$ such that $s\psi$ is a $\Ga_\rho$-critical linear form.
\item If $\psi$ is $\Ga_\rho$-critical, then 
$\delta_\psi=1.$
\end{itemize}
\end{lem}
\begin{proof}Set $s_0:=\inf\{s\geq 0:s\psi\geq\Phi_\rho\}$; we have $s_0\in(0,\infty)$ by \cite[{Thm. 4.2.2}]{Qu1}.
It follows that
$s_0\psi\geq\Phi_\rho$ and $s_0\psi(u)=\Phi_\rho(u)$ for some $u\in\frak a^+$ with $\norm{u}=1$, by the upper semi-continuity of $\Phi_\rho$ \cite[Lem. {3.}1.7]{Qu1}. In particular, $s_0\psi$ is $\Ga_\rho$-critical and the first assertion follows. The second assertion follows from \cite[Lem.{3.}1.3]{Qu1}.\end{proof}

\subsection*{Patterson-Sullivan measures}
Fix $o=(0,1)\in \bH^3$. By abuse of notation, we set 
$$
o=(o,\cdots,\,o)\in \prod_{i=1}^d \bb H^3.
$$ 
For $\xi=(\xi_1, \cdots, \xi_d)\in \hat \c^d$ and $g=(g_1, \cdots, g_d)\in G$,
the \textit{vector-valued Busemann map} is defined as
$$\beta_\xi(go, o)= (\beta_{\xi_1}(g_1 o, o), \cdots, \beta_{\xi_d}(g_d o, o))\in \frak a,$$
where  $\{\xi_i(t):t\ge 0\}$ is a geodesic ray in $\bH^3$ with $\lim\limits_{t\to+\infty}\xi_i(t)=\xi_i$ and
$$\beta_{\xi_i}(g_i o, o)=\lim_{t\to +\infty} \mathsf d(g_i o, \xi_i(t))-\mathsf d(o, \xi_i(t)).$$

Given a linear form $\psi\in\frak a^*$, a Borel probability measure $\nu$ supported on $\La_\rho$ is called a $(\Gamma_\rho,\psi)$-\textit{Patterson-Sullivan} ($\PS$) measure if 
for all $\ga\in \Ga_\rho$ and $\xi\in \cal F$,
$$\frac{d\ga_*\nu}{d\nu}(\xi)= e^{-\psi(\beta_{\xi}(\ga o, o))}.$$

We will say that $\nu$ is a $\Ga_\rho$-$\PS$ measure if it is a $(\Ga_\rho,\psi)$-$\PS$ measure for some $\psi\in\frak a^*$.
Extending the Patterson-Sullivan theory for rank one groups (\cite{Pa}, \cite{Su}), Quint \cite{Qu2} constructed  a $(\Ga_\rho,\psi)$-$\PS$ measure for each $\Ga_\rho$-critical linear form $\psi\in\frak a^*$ (see \cite{Bu} for earlier works on this).
As $\Ga_\rho$ is a \Zar Anosov subgroup of $G$, the following is a special case of \cite{LO}{.}
\begin{lem}\cite[{Thm. 1.1 and Thm. 4.3}]{LO}\label{lem.1-1}
For each $u\in\op{int}\cal L_\rho$, there exists a unique $\Ga_\rho$-critical linear form $\psi_u\in\frak a^*$ such that $\psi_u(u)=\Phi_\rho(u)$, and a unique $(\Ga_\rho,\psi_u)$-$\PS$ measure $\nu_{\psi_u}$.
The maps $u\mapsto \psi_u$ and $u\mapsto \nu_{\psi_u}$ give bijections among
$$
\{u\in\op{int}\cal L_\rho: \norm{u}=1\}\leftrightarrow\{\Ga_\rho\text{-critical linear forms}\}\leftrightarrow\{\Ga_\rho\text{-}\PS\text{measures}\}.
$$
\end{lem}

\section{Properties of admissible torus packings} \label{sec:adm}

\subsection*{Notations}
We will be using the following notations throughout the paper:

For $z=(z_i)_{i=1}^d\in\bb C^d$, set
\begin{equation}\label{eq.nz}
n_z=\left(\smatr{1}{z_1}{0}{1}, \cdots, \smatr{1}{z_d}{0}{1}\right)\in G.
\end{equation}
We also define the following subgroups:
$$ \;\;  N=\{n_z:z\in
{
\bb C^d
}
\},\;\;  \check N=\{n_z^t:z\in 
{
\bb C^d
}
\},$$
$$K=\prod_{i=1}^d \op{PSU}(2) \;\; \text{and} \;\;   H=\prod_{i=1}^d \left(\op{PSU}(1,1)\cup \smatr{0}{1}{-1}{0} \op{PSU}(1,1)\right), $$
where 
$$\op{PSU}(2)=\{\smatr{a}{ b}{ -\bar b}{ \bar a}:|a|^2+|b|^2=1\} \text{ and } \op{PSU}(1,1)=\{\smatr{a}{ b}{ \bar b}{ \bar a}:|a|^2-|b|^2=1\} .$$

We  set $$M= \{ \left( \smatr{e^{i\theta_1}}{0}{0}{e^{-i\theta_1}},\cdots,  \smatr{e^{i\theta_d}}{0}{0}{e^{-i\theta_d}}\right) :\theta_1, \cdots, \theta_d\in \br\};$$
note that $M$ is equal to the centralizer of $A$ in $K$.

Let $\cal C$ denote the space of all circles in $\hat{\bb C}$ (recall that a  
 { union of line and $\{\infty\}$} is considered as a circle with infinite radius) and $\cal T=\cal C\times \cdots\times \cal C$ the space of all tori in $\prod_{i=1}^d\hat{\bb C}$.
{
Under the identification made in \eqref{eq.DFT}, we may consider a torus as an element of $\cal T$, and a torus packing with a subset of $\cal T$.
}

\subsection*{$H$-orbits corresponding to
admissible torus packings} Throughout the paper, we fix the following torus
$$
 T_0=(C_0,  \cdots , C_0 )\in \cal T
$$
where $C_0=\{|z|=1\}$ is the unit circle centered at the origin.
Note that 
$$H=\op{Stab}_G(T_0),\quad\text{and}\quad K=\op{Stab}_G(o).$$
{Since $G$ acts transitively on $\cal T$, we can endow $\cal T\simeq G/H$ with the quotient topology on $G/H$.
Similarly, the topology on $\cal C$ will be induced from $\op{PSL}(2,\bb C)/\op{PSU}(1,1)$.
}

 We call a torus $T=(C_1, \cdots, C_d)$ $\Ga_\rho$-admissible
if for each $1\le i\le d$,
\begin{itemize}
    \item
{ $\{\rho_i(\ga)C_i\in\cal C:\ga\in\Ga\}$} is a locally finite circle packing;
\item $f_i(C_1\cap \La_\Ga) =C_i \cap \La_{\rho_i(\Ga)} .$
\end{itemize}

\begin{Def}\label{in-ad2}  A torus packing $\cal P{\subset\cal T}$ is called $\Ga_\rho$-admissible if
\begin{itemize}
    \item 
$\cal P$ consists of finitely many $\Ga_\rho$-orbits of $\Ga_\rho$-admissible tori;\item 
$\cal P$ is locally finite in the sense that no infinite sequence of tori in $\cal P$ converges to a torus. 
\end{itemize}
\end{Def}

The following lemma is rather standard (see for instance \cite[Lem. 3.2]{OS}.)

\begin{lem}\label{equiv} The followings are equivalent:
\begin{enumerate}  
    \item The torus packing $\Ga_\rho T_0 {\subset \cal T}$ is locally finite;
    \item The inclusion map $f : \Ga_\rho\cap H\ba H \to \Ga_\rho\ba G$ is proper;
    \item $\Ga_\rho \ba \Ga_\rho H$ is closed in $\Ga_\rho\ba G$.
    
\end{enumerate}
\end{lem}

 \begin{prop} \label{prop.UP}\label{Bounded}
If $\cal P=\Ga_\rho T_0$ is $\Ga_\rho$-admissible, then for any bounded subset $\cal O\subset\Ga_\rho\ba G$, the subset 
\be \label{bdd} 
\{[h]\in \Ga_\rho\cap H\ba H : [h] A^+\cap \cal O \neq\emptyset\}
\ee 
is bounded.
\end{prop}

\begin{proof}
Suppose not. Then there exist sequences $g_i\in \Ga$, $(h_{i,1}, \cdots, h_{i,d})\in H$, and 
$(t_{i,1}, \cdots, t_{i, d})\in \frak a^+$ such that ${(\Ga_\rho\cap H)}(h_{i,1},\cdots,h_{i,d})\to\infty$ in $\Ga_\rho\cap H \ba H$ as $i\to\infty$ and for each $1\le j\le d$, 
\begin{equation}\label{eq.ij}
    s_{i,j}:=\rho_j(g_i) h_{i,j}\smatr{e^{t_{i,j}/2}}{0}{0}{e^{-t_{i,j}/2}}
\end{equation}
is a bounded sequence {in $\op{PSL}(2,\bb C)$}.

Let  $H_0=\op{Stab}_{\PSL_2(\c)}(C_0)$ and 
$\mathsf D$
be a Dirichlet fundamental domain for the action of $\Ga\cap H_0$ on the convex hull $\widehat C_0\subset \bH^3$ of $C_0$.
By the admissibility hypothesis,  $\Ga C_0$ is a locally finite circle packing.
Hence the inclusion map $\Ga\cap H_0\ba \widehat C_0\to \Ga\ba \bH^3$ is
a proper map.
Since $\Ga$ is \Conv, it follows that $\partial  
\mathsf D
\cap \La_{\Ga}=\emptyset$ \cite[Prop. 5.1]{OS3} {where $\partial D:=D\cap C_0\subset\hat{\bb C}$ denotes the boundary at infinity of $D$}.

By replacing $h_{i,1}$ with an element of $(\Ga\cap H_0)h_{i,1}$ and modifying $g_i$ if necessary, we may assume that $h_{i,1}o\in  
\mathsf D
$.
Since ${(\Ga_\rho\cap H)}(h_{i,1},\cdots,h_{i,d})\to\infty$ in $\Ga_\rho\cap H \ba H$ as $i\to\infty$, we must have $h_{i,\ell}\to\infty$ in $H_0$ for some $1\leq\ell\leq d$.
By \eqref{eq.ij} and by the assumption that the sequence
$\{s_{i,j}:i=1,2, \cdots \}$ is bounded for each $1\le j\le d$,
 we have 
\begin{equation}\label{eq.ijo}
\xi_j:=\lim_{i\rightarrow\infty}h_{i,j}\smatr{e^{t_{i,j}/2}}{0}{0}{e^{-t_{i,j}/2}}o=\lim_{i\rightarrow\infty}\rho_j(g_i^{-1})s_{i,j}o \in  \La_{\rho_j(\Ga)}.
\end{equation}
It follows from the $\rho_j$-equivariance of $f_j$
 that
 $\xi_j =f_j (\xi_1)$ for each $1\le j\le d$.
{
We will need the following general fact from hyperbolic geometry: 
\textit{for any sequence $h_i\in H_0$ and $t_i\geq 0$ $(i\in\bb N)$, the sequence
\begin{equation}\label{eq.seq}
\left\{h_{i}\smatr{e^{t_{i}/2}}{0}{0}{e^{-t_{i}/2}}o\in\bb H^3:i\in\bb N\right\}    
\end{equation}
accumulates on $C_0$ if and only if $\{h_i\in H_0:i\in\bb N\}$ is unbounded. 
In this case, \eqref{eq.seq} shares the same limit point with $\{h_{i}o\in\bb H^3:i\in\bb N\}$ along any of its convergent subsequence.
}

Now,
} 
since $h_{i,\ell}\to\infty$, 
it follows from \eqref{eq.ijo} {and the above fact} that $\xi_\ell\in C_0\cap\La_{\rho_\ell(\Ga)}$.

Since
$C_0\cap \La_\Ga =f_\ell^{-1} (C_0 \cap \La_{\rho_\ell(\Ga)})$ by the assumption that $\scrP$ is $\Ga_{\rho}$-admissible, we have $\xi_1=f_\ell^{-1}(\xi_\ell) \in C_0\cap\La_{\Ga}$.
{
By \eqref{eq.ijo} and the previous fact from hyperbolic geometry, this implies that $h_{i,1}$ is unbounded and $h_{i,1}o\to\xi$ as $i\to\infty$.}
On the other hand, since $h_{i,1}o\in  
\mathsf D
$, we have $\xi_1\in \partial 
\mathsf D
$.
Hence $\xi_1\in \partial 
\mathsf D
\cap \La_\Ga$;
 this yields a contradiction since $\partial  
\mathsf D
 \cap\La_{\Ga}=\emptyset$.
\end{proof}

\begin{prop}\label{lem.supp}
If $\cal P=\Ga_\rho T_0$ is $\Ga_\rho$-admissible, then the following hold:
\begin{enumerate}
    \item 
    the set
    $$\{[h]\in \Ga_\rho\cap H \ba H : h P\in \La_\rho \}$$
is compact;
    \item
    for any bounded subset $S\subset G$ and any closed cone $\cal E\subset \frak a^+$
such that $\cal E\cap \L_\rho=\{0\}$, we have
$$\# (\big(H\ba H\Ga_\rho)\cap (H \ba H
\exp(\cal E)S  \big))<\infty.$$
\end{enumerate}

\end{prop}

To prove the proposition we will use the following lemma, {which is equivalent to \cite[Prop. 7.4]{LO} in view of the characterization of the limit cone $\cal L_\rho$ as an asymptotic cone of $\{\mu(\ga):\ga\in\Ga_\rho\}$ given in \cite[Thm. 1.2]{B}}.
\begin{lem}[Uniform conicality of $\La_\rho$] \cite[Prop. 7.4]{LO} \label{lem.con}
There exists a compact subset $\cal Q\subset G$ such that the following holds: for any $g\in G$ with $g P\in\La_\rho$ and any closed
{convex} 
cone $\cal D\subset\op{int}\frak a^+{\cup\{0\}}$ whose interior contains $\cal L_\rho-\{0\}$, we can find sequences $\ga_i\in\Ga_\rho$ and $\log a_i\to\infty$ in $\cal D$ such that 
$$\ga_i g a_i\in \cal Q \quad \text{ for all $i\ge 1$}. $$
\end{lem}

\medskip 
\subsection*{Proof of Proposition \ref{lem.supp}}
Let $\cal Q\subset G$ be as in Lemma \ref{lem.con}.
Choose any closed {convex} cone $\cal D\subset\inte\frak a^+{\cup\{0\}}$ whose interior contains $\cal L_\rho-\{0\}$.
Since the inclusion map $\Ga_\rho\cap H\ba H\to \Ga_\rho\ba G$
is a proper map,
 Lemma \ref{lem.con} implies that 
 $$\{[h]\in \Ga_\rho\cap H \ba H : h P\in \La_\rho \}\subset \{[h]\in \Ga_\rho\cap H\ba H : [h] \exp \cal D \cap \cal Q \neq\emptyset\}.$$

By Proposition \ref{Bounded}, the subset on the right-hand side is bounded. Therefore (1) follows.

Suppose (2) is false. Then there exists a bounded subset $S\subset G$ and infinite sequences $t_i\in \cal E$, $t_i\rightarrow\infty$, $\ga_i\in  \Ga_\rho$, $ h_i\in  H$, and $s_i\in S
$ such that
$$\ga_i=h_i 
a_{t_i}s_i, $$
and $ H\ga_i\neq H\ga_j$ for $i\neq j$.
Since the image of $\ga_i^{-1}h_ia_{t_i}=s_i^{-1}\in S^{-1}$ under the projection $G\to \Ga_\rho\ba G$ is bounded, it follows again from Proposition \ref{Bounded} that there exists a sequence $\delta_i\in \Ga_\rho\cap H $ such that the sequence $\tilde h_i:=\delta_i h_i$ is bounded.
Set $\tilde \ga_i:=\delta_i \ga_i$.
Note that $H\tilde \ga_i=H \ga_i$ and $\tilde \ga_i =\tilde h_i a_{t_i}s_i\in \Ga_\rho$.
Since both $\tilde h_i $ and $s_i$ are bounded, the sequences $t_i$ and $\mu(\tilde \ga_i)$ are within bounded distance of each other. Now using the fact that $\cal L_\rho$ is the asymptotic cone of $\{\mu(\ga):\ga\in
\Ga_\rho \}$, and $\cal E\cap\cal L_\rho=\{0\}$, we have $t_i\not\in\cal E$ for all sufficiently large $i$, which is a contradiction.

\subsection*{Closedness of  $\Ga_\rho T_0$}
The following proposition says that local finiteness of  
 $\Ga_\rho T_0 { \subset\cal T}$
is a consequence of the local finiteness of $\Ga C_0 {\subset \cal C}$ when $T_0$ is an admissible torus with $\# (C_0\cap \La_\Ga) \ge 3$. 
\begin{prop}\label{ccc}
Let $\Ga C_0$ be closed in $\cal C$ with $\# (C_0\cap \La_\Ga) \ge 3$. 
If $f_i(C_0\cap \La_\Ga) =C_0 \cap \La_{\rho_i(\Ga)}$ for each $1\le i\le d$, then
 $\Ga_\rho T_0$ is closed in $\cal T$ and $\rho_i(\Ga) C_0$ is closed in $\cal C$ for all $2\le i\le d$.
\end{prop}
\begin{proof}
Suppose that a sequence $T_n=(\rho_1(g_n) C_0,\rho_2(g_n)C_0,\cdots,\rho_d(g_n)C_0)$ converges to some torus $T=(C_1,C_2,\cdots,C_d)$ for $g_n\in\Ga$.
We need to show that $T\in\Ga_\rho T_0$.
Since $\Ga C_0$ is closed and hence locally finite by Lemma \ref{equiv}, we may assume that for all $n\ge 1$, $g_nC_0=C_1$ by throwing away finitely many $g_n$'s (recall $\rho_1=\text{id})$.
Observe that $\rho_i(g_n)f_i(C_0\cap \La_\Ga)=f_i(g_n(C_0\cap\La_\Ga))=f_i(C_1\cap\La_\Ga)$ by \eqref{eq.EQV}.
On the other hand $f_i(C_0\cap\La_\Ga)=C_0\cap\La_{\rho_i(\Ga)}$ and it contains at least 3 distinct points. Since two circles sharing three distinct points must be equal to each other,
we get $\rho_i(g_n)C_0=C_1$ for all $1\le i\le d$ and all $n$.
It follows that $T_n=T=T_0$ for all $n$, proving the first claim. The second claim can be proved similarly.
\end{proof}

Although we won't be using the following proposition in the rest of our paper, {it} is of independent interest {and} extends the 
analogous
fact for \Conv groups for $d=1$.
\begin{prop} Let $T$ be a torus and $H_T$ be the stabilizer of $T$ in $G$.
Suppose $\Gr T$ is closed with $\# (T\cap \La_\rho) \ge 3$. 
Then $\Gr\cap H_T$ is a non-elementary Anosov subgroup
and $$T\cap \La_\rho =\La_{\Gr\cap H_T}.$$
\end{prop}

\begin{proof} 
Without loss of generality, we may assume {that} $H_T$ is the product of $\PSL_2(\br)$'s.
We use the characterization of an Anosov subgroup as a subgroup of $G$ satisfying {the properties of} Regularity, Conicality, Antipodality, 
shown in
\cite[{Thm. 1.1}]{KLP1}.  Since $H_T\cap \Gr$ is a subgroup of an Anosov subgroup $\Gr$, $H_T$ contains $A$ and $H_T/(H_T\cap P)\subset G/P$ is the Furstenberg boundary of $H_T$,
the regularity and antipodality are immediate.

We deduce the conicality as follows. 
Let $\xi\in T\cap \La_\rho$. We can choose $h\in H_T$ such that $h(H_T\cap P)=\xi$.
Since $\Gr$ is Anosov, $\xi$ is a radial limit point of $\Gr$, that is, there exist $a_n\to \infty$ in $A^+$ and $\delta_n\in \Ga_\rho$ such that
$\delta_n h a_n $ is bounded. Since the map $\Gr\cap H_T\ba H_T\to \Gr \ba G$ is proper by Lemma \ref{equiv} and $ha_n\in H_T$,
it follows that there exists $\tilde \delta_n\in\Gr\cap H_T$ that
$$\tilde \delta_n h a_n$$ is bounded.
This implies that $\xi=h (H_T\cap P)$ is a radial limit point of $\Gr\cap H_T$ in
$H_T/(H_T\cap P)$.
Hence we have shown that
$T\cap \La_\rho $ is equal to the set $\La^{\text{rad}}_{\Gr\cap H_T}$ of all radial limit points of $\Gr\cap H_T$.
Since $\La_{\Gr\cap H_T}\subset T\cap \La_{\rho}$, it follows that
$$\La_{\Gr\cap 
H_T
}=\La^{\text{rad}}_{\Gr\cap H_T}.$$
Hence $
\Gr\cap 
H_T
$ is conical. This proves that $\Gr\cap 
H_T
$ is Anosov. The hypothesis $\# (T\cap \La_\rho)\ge 3$ now implies that 
$\Gr\cap 
H_T
$
is non-elementary.
\end{proof}

\section{Torus counting function for admissible torus packings}\label{sec:count}
We write $\mathsf r(C)$ for the radius of a circle $C$.
Given a torus $T=(C_1,\cdots,C_d)\in \cal T$, we define its \emph{length vector} {$\vecv(T)\in\frak a\cup\{\infty\}$} by
$$\vecv(T)=-\big(\log \mathsf r(C_1), \cdots, \log \mathsf r(C_d) \big)$$
{if $\mathsf r(C_i)<\infty$ for all $1\leq i\leq d$, and $\vecv(T)=\infty$ otherwise.}

We will call a linear form $\psi\in \fa^*$ \textit{positive} if 
$$\text{$\psi>0$ on $\fa^+-\{0\}$.}$$

In the rest of this section, we  fix 
\begin{itemize}
    \item a $\Ga_\rho$-admissible torus packing
$\cal P=\Ga_\rho T_0$;
\item a positive $\Ga_\rho$-critical linear form $\psi\in\frak a^*$.
\end{itemize} 

\begin{Def}[Counting function] For  a bounded subset $E\subset{\bb C}^d$ and $R>0$, we set
\begin{align}\label{eq.N}
N_{R}(\cal P,\psi, E)=\#\left\lbrace T\in \cal P:  
\psi(\vecv(  T))< R  ,\; T \cap E \neq \emptyset \right\rbrace.
\end{align}
\end{Def} 

The local finiteness assumption on $\cal P$ together with the positivity hypothesis on $\psi$ guarantees that
\begin{lem} \label{fin5} For any bounded subset $E\subset{\bb C}^d$ and $R>0$,
$$N_{R}(\cal P,\psi, E)<\infty .$$
\end{lem}
\begin{proof}
By the local-finiteness of $\rho_i(\Ga)C_0$, there are only finitely many circles in $\rho_i(\Ga)C_0$
of radius bounded from below intersecting a fixed bounded set.
In particular,
$$
n_0:=\#\{T\in\cal P : \vecv(T)\not\in\frak a^+\text{ and }T\cap E\neq\emptyset\}<\infty.
$$
By the positivity hypothesis on $\psi$, we have
$c:=\inf_{v\in \fa^+, \|v\|=1} \psi (v) >0$ and hence
$\psi (v)\ge c \| v\|$ for all $v\in \fa^+$.
Hence 
\begin{align*} 
&N_{R}(\cal P,\psi, E)-n_0
\\&\le  \#\{T=(C_1,\cdots,C_d)\in \cal P: \sum_{i=1}^d |\log \mathsf r (C_i)|^2\le
R^2/c^2 \text{ and }T\cap E\neq\emptyset\}
\\ &\le \#\sum_{i=1}^d \{ C\in \rho_i(\Ga) C_0 : e^{-R/c}  \le 
\mathsf r (C_i) 
\text{ and } C\cap \pi_i(E)\ne \emptyset\} 
\end{align*}
where $\pi_i(E)$ denotes the projection of $E$ to the $i$-th factor $\hat \c$.
The last quantity is finite by the local-finiteness of $\rho_i(\Ga)C_0$.
 This proves the claim.
\end{proof}

We will introduce a subset $\tilde B_\psi(E,R)
\subset H\ba G$ and explain how 
$N_{R}(\cal P,\psi,  E)$ is related to the number of $\Ga_\rho$-orbits in the set $
\tilde B_\psi (E,R)
$.

\subsection*{Definition of $\tilde B_\psi(E,R)$}
For $R>0$, we define
$$
 A^+_{\psi,R}= \{a_t \in A^+:\psi(t)<R\},
$$
where $a_t$ is defined as in \eqref{at}. As $\psi$ is positive, $A^+_{\psi, R}$ is bounded.

{For any subset $E\subset\bb C^d$, we define
$$
N_E=\{n_z\in N:z\in E\}
$$
where $n_z$ is defined as in \eqref{eq.nz}.
For any $\e>0$, set
\begin{equation}\label{eq.epm}
    E^-_{\e}:=\bigcap_{\|\vecw\| < \epsilon} E+\vecw,\quad \text{and }\quad E^+_{\e}:=\bigcup_{\|\vecw\| < \epsilon} E+\vecw .
\end{equation}
}
\begin{Def} For any bounded $E\subset \CC^d$ and $R>0$, we define the following bounded subset of $H\ba G$ by
\be\label{BER0} 
\tilde B_\psi(E,R)
:=H\ba HK  
A^+_{\psi, R}
N_{-E} \subset H\ba G .\ee 
\end{Def} 
The following proposition  allows us to reformulate the counting problem in terms of the sets $\tilde  B_\psi (E_\e^\pm,R)$ (cf. \ \cite[Proposition 3.7]{OS}):
For $\e>0$, set
\be\label{qzero} q_0(\cal P, E, \e):=\#\{T=(C_1,\cdots,C_d)\in \cal P: \sum_{i=1}^d \mathsf r (C_i)^2> \e^2/4 \text{ and }T\cap E\neq\emptyset\}.\ee 
The finiteness of $q_0(\cal P, E, \e)$ can be seen 
{
as in}
the proof of Lemma \ref{fin5}.

\begin{prop}\label{BETcount1} Let $E\subset \CC^d$ be a bounded subset.
For any $\e>0$ small enough and any $R>0$, we have
$$
\#\big([e]\Ga_\rho \cap \tilde  B_\psi (E_\e^-,R)\big)-q_0 \leq N_{R}(\cal P,\psi, E)\leq \#\big([e]\Ga_\rho\cap \tilde  B_\psi (E_{\e}^+,R)\big)+q_0
$$
where $q_0=q_0(\cal P, E, \e)$.
\end{prop}
\begin{proof} 
{
Let $\hat T_0=\hat C_0\times\cdots\times\hat C_0$.
}
Note that
\begin{align}\label{eq.z1}
&\#\big([e]\Ga_\rho \cap H\ba HKA_{\psi,R}^+N_{-E_\e^{\pm}}\big)\notag\\
&=\#\lbrace \gamma \in \Gamma_\rho\cap H \ba \Ga_\rho\::\: H\ga \cap KA_{\psi,R}^+N_{-E_\e^{\pm}}
\neq\emptyset
\rbrace\notag
\\&=\#\lbrace \gamma \in \Ga_\rho/\Gamma_\rho\cap H\: :\,\gamma HK \cap N_{E_\e^{\pm}}(A_{\psi,R}^+)^{-1}K \neq \emptyset \rbrace\notag
\\&=\#\lbrace \gamma  { T_0}\in\scrP\,:\,\gamma \hat{T_0} \cap N_{E_\e^{\pm}}(A_{\psi,R}^+)^{-1}o \neq \emptyset \rbrace. 
\end{align}
Observe that
{
for $z=(z_i)_{i=1}^d\in\bb C^d$, $t=(t_i)_{i=1}^d\in\bb R^d$ and $o=(0,1)_{i=1}^d\in \prod_{i=1}^d\bb H^3$, we have $n_za_t o=(z_i,t_i)_{i=1}^d\in\prod_{i=1}^d\bb H^3$. 
Hence,
}
if  $\ga T_0\in\cal P$, $ \ga \hat{T_0}\cap N_{E_\e^-} (A_{\psi,R}^+)^{-1}o\neq\emptyset$ and  
$
\sum_{i=1}^d \mathsf{r}(\rho_i(\ga)C_0)^2\leq \e^2/4,
$
then $\ga T_0\cap E\neq\emptyset$ and
$  \psi(\vecv(\ga T_0))<R$.
This observation combined with \eqref{eq.z1} gives the lower bound in the statement of the proposition. Similarly, if $\ga T_0\in\cal P$ satisfies $\ga T_0\cap E\neq\emptyset$, $\psi(\vecv(\ga T_0))< R$ and 
$
\sum_{i=1}^d \mathsf r(\rho_i(\ga)C_0)^2\le \e^2/4,
$
then
$\ga T_0\subset E_{\e^+}$ and hence $\ga \hat{T_0}\cap  N_{E_\e^+} (A_{\psi,R}^+)^{-1}o \ne \emptyset$.
This combined with \eqref{eq.z1} gives the upper bound, proving the proposition.
\end{proof}

\subsection*{Definition of $B_\psi(E,R)$}
Let $\cal D\subset \op{int}\frak a^+$ be any closed cone such that
\be
\label{fixD} \inte\cal D\supset \cal L_\rho-\{0\}.
\ee 
Throughout the section we fix one such $\cal D$ and set, for any $R>0$,
\be\label{ddd} 
  D:=\exp \cal D\quad \text{ and } \quad D_{\psi, R}=D\cap A^+_{\psi,R}.
\ee

Analogously to $\tilde B_\psi (E,R)$, we now define
\begin{equation}\label{BER}
B_\psi(E,R)=B_{D,\psi}(E,R):=H\ba HKD_{\psi, R}
N_{-E} \subset H\ba G.
\end{equation}

\subsection*{$B_\psi(E,R)$ in terms of $G=HA^+K$ decomposition} We will now express {the} set $B_\psi(E,R)$ in terms of the
generalized Cartan decomposition $G=HA^+K$ (cf. \cite[p. {439}]{GOS2}).
Given $\epsilon>0$ and a subset $W\subset G$, let $W_{\epsilon}
$ denote the intersection of $W$ and the $\e$-ball around $e$ in $G$.  
\begin{lem}\cite[Proposition 4.2]{OS}\label{GEO}
For $d=1$, we have
\begin{enumerate}
\item If $a_t\in HKa_sK$ for some $s\geq 0$, then $|t|\leq s$.
\item For any $\epsilon>0$, there exists $R_1(\epsilon)>0$ such that 
$$\lbrace k\in K\,:\, a_tk\in HKA^+\; \mathrm{for\,some\,}t>R_1(\epsilon)\rbrace \subset K_{\epsilon} M.$$
\end{enumerate}
\end{lem}

We set
\be\label{Xe} \cal X_\e:=\{a_t\in A^+: \min_{1\le i\le d} t_i \le R_1(\e/\sqrt d)\},\ee  that is 
the closed $R_1(\e/\sqrt d)$-neighborhood of $\partial A^+$, where $R_1(\e/\sqrt{d})>0$ is the constant as given
in Lemma \ref{GEO}(2).

We deduce the following{.}
\begin{lem}\label{HKA}

For any $\e>0$ and $R>0$,
$$
K A^+_{\psi, R} \subset H(A^+_{\psi, R}-\cal X_\e)K_\e\cup H\cal X_\e K.
$$
\end{lem}

\begin{proof}
For any $k=(k_1,\cdots,k_d)\in K$ and $a_t\in A^+$, using the decomposition $G=HA^+K$, we can find  $h=(h_1,\cdots,h_d)\in H$, $a_s\in A^+$ and $\ell=(\ell_1,\ldots,\ell_d)\in K$ such that
\begin{equation}\label{eq.rec}
k_i \smatr{e^{t_i}/2}{0}{0}{e^{-t_i/2}} =h_i \smatr{e^{s_i}/2}{0}{0}{e^{-s_i/2}} \ell_i 
\end{equation}
for all $i=1,\ldots,d$, where $t=(t_i)_{1\leq i\leq d}$ and $s=(s_i)_{1\leq i\leq d}$.
From Lemma \ref{GEO}(1), we then have $s_i \leq t_i$.
Since $\psi|_{ \frak a^+}\geq 0$, we have 
$$
\psi(s)\leq\psi(t).
$$
Hence if $a_t\in A^+_{\psi, R}$,  then we have $a_s\in A^+_{\psi, R}$. Furthermore, if $a_s\not\in\cal X_\e$, we have $s_i>R_1(\e/\sqrt{d})$ for each $i$ and hence $\ell\in K_{\e} M$ by Lemma \ref{GEO}(2).
{Since $K_\e M=MK_\e$ and $M\subset H$,}
this proves the lemma.  
\end{proof}

\subsection*{Further refinement}
The  following lemma appears in \cite[Prop. 4.7]{OS} for the case $d=1$, and this implies the general $d$-case as the computations can be reduced to each component{.}
\begin{lem}\label{WAVE} \cite[Prop. 4.7]{OS}
 There exists ${\ell'}\ge 1$ such that for all small enough $\epsilon>0$, {and} $a\in A^+$
$$ a K_\e 
M
\subset H ( a A_{\ell{'}\e} ) N_{\ell{'}\e}  .$$
\end{lem}

Using Lemma \ref{WAVE}, we obtain{.}
\begin{lem}\label{lem.XE}
For any $\epsilon>0$ and a bounded subset $E\subset\bb C^d$, there exists a compact subset $Z=Z(E, D, \e)\subset H\ba G$ such that
$$
\tilde B_\psi(E,R)\subset  H\ba HD_{\psi, R}A_{{\ell'}\e}N_{-E_{{\ell'}\e}^+}\cup H\ba H(A^+-D)KN_{-E} \cup Z,
$$
\end{lem}
\begin{proof} Since $\cal D$ is a closed cone contained in
$\inte \fa^+$, $D\cap \cal X_\e$ is a compact subset. Therefore
 $Z:=H\ba H(D\cap \cal X_\e)KN_{-E}$ is a bounded subset of $H\ba G$.

Note that by Lemma \ref{HKA}
and  Lemma \ref{WAVE}
\begin{align}\label{eq.aa}
K A^+_{\psi, R}&\subset H (A^+_{\psi, R}-\cal X_\e) K_\e \cup H\cal X_\e K\notag\\
&\subset HD_{\psi, R}K_\e \cup H(A^+-D)K \cup H(D\cap\cal X_\e)K 
\\ 
& \subset H D_{\psi, R} A_{{\ell'}\e} N_{{\ell'}\e}\cup H(A^+-D)K \cup H(D\cap\cal X_\e)K {.}
\end{align}
 The claim now follows from the definition of
$\tilde B_\psi(E,R)$.
\end{proof}

\begin{cor}\label{cor.C}
 For any $\e>0$,
 there exist $q_1=q_1(E, D, \e)>0$ and ${\ell'} ={\ell'} (\psi)$ such that 
$$
\#([e]\Ga_\rho\cap B_\psi(E,R))\leq \#([e]\Ga_\rho\cap \tilde B_\psi(E,R))\leq \#([e]\Ga_\rho\cap B_\psi({E_{\ell'\e}^+},R+ {{\ell'}\e}))+q_1.
$$
\end{cor}
\begin{proof}
The first inequality is trivial.
{For} the second inequality,
{choose}
a slightly smaller closed cone $\cal D_0\subset \inte \cal D$ such that
$\L_\rho-\{0\}\subset \inte \cal D_0$ and
{set}
$\cal E=\overline{\fa^+-\cal D_0}$
. 
{Note that $D_{\psi,R}-(D_0)_{\psi,R}A_{\ell'\e}$ is a bounded set and hence applying Lemma \ref{lem.XE} to the cone $\cal D_0$ shows
$$
\tilde B_\psi(E,R)\subset H\ba HD_{\psi,R}N_{-E_{\ell'\e}^+}\cup H\ba H\cal E KN_{-E}\cup Z'
$$
for some compact set $Z'\subset H\ba G$.
Applying Proposition \ref{lem.supp} with $S=KN_{-E}$ gives the desired conclusion.
}
.
\end{proof}

\section{Mixing and  equidistribution with uniform bounds}\label{uSec}
We fix a positive $\Ga_\rho$-critical linear form $\psi\in\frak a^*$ and the $(\Ga_\rho, \psi)$-PS measure $\nu_\psi$ given by Lemma \ref{lem.1-1}.
In this section, we recall the results of \cite{CS} and \cite{ELO} on  mixing (Proposition \ref{prop.mixH}) and equidistribution (Proposition \ref{prop.mixH2}), with emphasis placed on their uniformity aspects that are crucial in our application. 
\subsection*{Burger-Roblin measures $m^{\BR}$ and $m^{\BR_*}$}

Recall that $P=\op{Stab}_G(\infty,\cdots,\infty)$ denotes the product of upper triangular subgroups. We also denote by $ \check P=\op{Stab}_G(0,\cdots,0)$
the product of lower triangular subgroups.

For $g\in G$, its visual images are defined by
$$g^+ :=gP\in \cal F \quad \text{ and }\quad g^-:= g\check P\in \cal F.$$
Let $\cal F^{(2)}$ denote the unique open $G$-orbit in $\cal F\times\cal F$ under the diagonal action, that is,
$
\cal F^{(2)}=\{(g^+, g^-): g\in G\}
$.

The map 
$$
gM\mapsto (g^+, g^-, b=\beta_{g^-}(o, go))
$$
gives a homeomorphism
 $G/M\simeq  \F^{(2)}\times \fa $, called the Hopf parametrization of $G/M$.
We define a locally finite Borel measure $\tilde m_\psi^{\BR}$ on $G/M$ as follows:
for $g=(g^+, g^-, b)\in \F^{(2)}\times \mathfrak a$,
\begin{equation}\label{eq.BMS0}
d\tilde m_\psi^{\BR} (g)=e^{
\psi(\beta_{g^+}(o, go))+2\sigma( \beta_{g^-} (o, go )) } \;  d\nu_{\psi} (g^+) dm_o(g^-) db,
\end{equation}
  where $m_o$ is the unique $K$-invariant probability measure on $\cal F$, $db $ is the Lebesgue measure on $\mathfrak a$, and $\sigma$ is the linear form on $\frak a$ defined by 
  \begin{equation}\label{eq.sum}
\sigma(t_1,\cdots,t_d)=t_1+\cdots+t_d.
  \end{equation}  

By abusing notation slightly, we will also use $\tilde m_\psi^{\BR}$ to denote the corresponding $M$-invariant measure on $G$ induced by $\tilde m_\psi^{\BR}$. The measure $\tilde m_\psi^{\BR}$ is left $\Ga_\rho$-invariant and induces an $\check N$-invariant locally finite measure on $\Ga_\rho\ba G$, which we denote by $m_\psi^{\BR}$.

Similarly, 
but with a different parameterization $g=(g^+,g^-,b=\beta_{g^+}(o,go))$, we define the following $N$-invariant locally finite Borel measure {on $G$}:
\be\label{dualb} d\tilde m_\psi^{\BR_*}(g)=e^{2\sigma (\beta_{g^+}({{o}}, g{{o}}))+\psi(\beta_{g^-} ({{o}}, g{{o}})) } \; d m_o(g^+) d\nu_{\psi} (g^-)  db. \ee

We have the following decomposition (see (4.8) of \cite{ELO}){.}
\begin{lem}\label{PNdecomp}
For $f\in C_c(P\check N)$,
\begin{equation*}
\tilde{m}_\psi^{\mathrm{BR}}(f)=\int_{NAM}\left(\int_{\check N}f(nam \check n)\,d\check n\right)e^{
-\psi(\log a)}
{e^{\psi(\beta_{n^-}(o,n o))}}
\,dm\,da\,d\nu_\psi({n^-}),
\end{equation*}
where $dm$, $da$, $d\check n$ denote the Haar measures for $M$, $A$, $\check N$, respectively.
\end{lem}
We note that in Lemma \ref{PNdecomp}, $dm$ is normalized to be a probability measure on $M$, $da$ is normalized to be compatible with the restriction of the Killing form on the lie algebra of $A$, and $d\check{n}$ is equivalently given by the density $\check{n} \mapsto e^{2\rho(\beta_{\check{n}^+}(o, \check{n}o))} d\nu_0(\check{n}^+)$ where $\nu_0$ denotes the unique $K$-invariant probability measure on $\mathcal F$.

\subsection*{Patterson-Sullivan measure $\mu_{\Ga_\rho\cap H \ba H,\psi}^{\PS}$ \cite[Definition 8.7]{ELO}}

We define a measure $\mu_{ H,\psi}^{\mathrm{PS}}$ on ${H}$ as follows: for $\phi\in C_c(H)$, let
$$
 \mu_{{H},\psi}^{\mathrm{PS}}(\phi)=  \int_{ h\in H/H\cap P} 
  \int_{p\in H\cap P}    \phi( hp ) e^{
\psi(\beta_{ h^+}({{o}},  hp{{o}}))}  \,dp
 \, d\nu_\psi( h^+),
$$ where $dp$ is a right-Haar probability measure on $H\cap P$ (note that $H\cap P$ is compact for the pair $(G,H)$ we are considering); for $h\in H/H\cap P$, $h^+$ is well-defined and independent of the choice of a representative.
The measure defined above is $\Gamma_\rho\cap H$-invariant: for any  $\gamma\in \Gamma_\rho\cap H$, $\gamma_*\mu_{ {H},\psi }^{\mathrm{PS}}=
 \mu_{{H},\psi }^{\mathrm{PS}}$.  Therefore, if $\Gamma_\rho\ba \Gamma_\rho H$ is closed in $\Gamma_\rho\ba G$,
$d\mu_{H,\psi}^{\PS}$ induces a locally finite Borel measure on $\Gamma_\rho\ba \Gamma_\rho H\simeq \Gamma_\rho\cap H\ba H$, which we denote by $\mu^{\PS}_{\Ga_\rho\cap H \ba H,\psi}$.

The \textit{skinning constant} of $\G_\rho\cap H\ba H$ with respect to $\nu_\psi$
is defined as the total mass:
\begin{equation}\label{eq.sk}
\op{sk}_{\Ga_\rho,\psi}(H):=|{|}\muPS_{\Ga_\rho\cap H\ba H, \psi}{|}|\in [0, \infty].
\end{equation}

\subsection*{Uniform mixing}
We fix the unique unit vector $u=u_\psi \in \inte \L_\rho$ such that
$$\psi(u)=\Phi_\rho(u)$$
provided by Lemma \ref{lem.1-1}.

Since the cone $\fa^+$ is contained in the closed half space $\{\psi\geq 0\}$ { and $\psi(u)>0$, $\frak a^+$}
can be parameterized by the map
\begin{align*}
    \RR_{\geq 0}\times 
\op{ker}\psi&\to \frak a\\
(s,w)&\mapsto su+\sqrt sw.
\end{align*}

The following mixing result is due to \cite[Thm. 3.4]{ELO2} and
\cite[Thm. 1.4 \& Thm. 1.5]{CS}:
the uniform bound as stated in the second part is crucial in our application as remarked before.

\begin{thm} \label{m11}\label{prop.mixH} 
There exists an inner product $\langle\cdot,\cdot\rangle_*$ on $\frak a$ and $\kappa, {\ell}>0$ such that for any $f_1, f_2\in C_c(\GaG)$ and $w\in \ker \psi$,
\begin{multline*}
 \lim_{s\to +\infty} s^{(d-1)/2} e^{(2\sigma -\psi) (su_\psi +\sqrt sw)}  \int_{\Ga_\rho\ba G} f_1(x \exp(su_\psi +\sqrt sw) ) f_2(x) dx\\
= \kappa\,e^{-{\ell}I(w)} m_\psi^{\BR} (f_1) m_\psi^{\BR_*} (f_2)
\end{multline*}
 where $I:\op{ker}\psi\to\bb R_{\ge 0}$ is given by  $
I(\vecw)=\tfrac{\|\vecw\|_*^2-\langle\vecw,
u_\psi
\rangle_*^2}{\norm{u_\psi}_*^2}.
$
Moreover, there exist $s_0, \ell >0$ and $C'=C'(f_1,f_2)>0$ such that for all $(s,\vecw)\in (s_0,\infty)\times\op{ker}\psi$ with $su_\psi +\sqrt s\vecw\in\fa^+$, we have:
\begin{equation}
\left| s^{(d-1)/2} e^{(2\sigma-\psi)(su_\psi +\sqrt sw)} \int_{\Ga_\rho\ba G} f_1(x \exp(su_\psi +\sqrt sw) ) f_2(x) dx \right| \le C'
e^{-\ell  I(w)}.
\end{equation}
\end{thm}

\subsection*{Uniform equidistribution}
Using Theorem \ref{m11}, the following equidistribution result can be obtained as in \cite[Proposition 8.11]{ELO} {and using a partition of unity argument for $\phi$}{.}

\begin{prop}\label{prop.mixH2} 
{Assume that $\Ga_\rho H$ is closed,} let $f\in C_c(\Gamma_\rho\ba G)$ and $\phi\in C_c({\Ga_\rho\cap H\ba}H)$. For any $w\in \ker\psi$, we have
\begin{multline}\label{eq.unif}
\lim_{s\to +\infty} s^{(d-1)/2}e^{(2\sigma- \psi)(su_\psi +\sqrt sw)} \int_{\Ga_\rho\ba \Ga_\rho H} f ([h]\exp(su_\psi +\sqrt sw)) \phi (h)\,dh \\
 = \kappa\, e^{-{\ell}I(w)}\,m_\psi^{\mathrm{BR}}(f)\, \mu_{{\Ga_\rho\cap H\ba}H,\psi}^{\PS}(\phi)
\end{multline}
 where $\kappa, {\ell} >0$ and $I :\op{ker}\psi\to\bb R_{\ge 0}$ are given {by} Theorem \ref{m11}.
Moreover, there exists  $C''=C''(f,\phi), {s_0}>0$ such that for all $(s,\vecw)\in (s_0,\infty)\times\op{ker}\psi$ with $su+\sqrt s\vecw\in\fa^+$, 
\begin{equation}\label{qqqq}
\left| s^{(d-1)/2} e^{(2\sigma-\psi)(su_\psi +\sqrt sw)} \int_{H} f ([h]\exp(su_\psi +\sqrt sw)) \phi (h)\,dh\right| <C'' 
e^{-\ell  I(w)}.
\end{equation}
\end{prop}

\section{The measure $\omega_{\psi}$}\label{omegaSec}
Fix a positive $\Ga_\rho$-critical linear form $\psi\in\frak a^*$ and the $(\Ga_\rho, \psi)$-PS measure $\nu_\psi$ given by Lemma \ref{lem.1-1}.
\begin{Def}\label{eq.om}
We define a locally finite Borel measure $\om_{\psi}=\om_{\Ga_\rho,\psi}$ on $\CC^d$ as follows: for all $f\in C_c(\bb C^d)$,
\begin{equation*}
    \om_{\psi} (f)
=\int_{\bb C^d} f(z)e^{\psi(\beta_{z}(o,n_{z}\cdot o))}\,d\nu_\psi(z
).
\end{equation*}
\end{Def}

For each small $\e>0$, let $\phi^{\e}\in C_c(N_{\e}A_{\e}M_{\e}{\check N}_{\e})$ be a  non-negative function such that $\int_G \phi^{\e}\,dg=1$ where $dg$ is a Haar measure on $G$ and {for any $z\in\bb C^d$}, set
$$\phi_{\vecz}^\e(g):=\int_M \phi^\e(gmn_{\vecz})\,dm,$$
where $dm$ is a probability Haar measure on $M$.

The main goal of this section is  to establish Corollary \ref{phibdd}, which roughly says 
$$ \int_{-E} m_\psi^{\BR}(\phi^{\e}_{\vecz})\,d\vecz\approx \om_{\psi} (E).$$

Let $E\subset \CC^d$ be a
{
fixed} 
bounded Borel set and  $\e>0$ be small enough so that $$ A_{\e}M_{\e}{\check N}_{\e}N_{-E}N_1\subset NAM{\check N}.$$

{
For all $z\in\bb C^d$,
}
define $\Phi_E^{\e}\in C_c({\bb C^d})$ by 
$$ \Phi_E^{\e}({z}):=\int_{N_{-E}MA{\check N}}\phi^{\e}(n_{ z}g)\,dg=\int_{N_{-E}A_{\e}M_{\e}
\check N_{\e}
}\phi^{\e}(n_{ z}g)\,dg.$$
Recalling the definition of $E_\e^\pm$ from \eqref{eq.epm}, we have
\begin{lem}\label{PhiE} For all 
{$z\in\bb C^d$},
$$ \mathbbm{1}_{E^-_{\e}}(n_{z}) \leq  \Phi_E^{\e}({z}) \leq \mathbbm{1}_{E^+_{\e}}(n_{z}).$$
\end{lem}
\begin{proof}
Trivially, $0\leq \Phi_E^{\e}({z})\leq 1$. If ${z}\in E^-_{\e}$, then $n_{z}^{-1}N_{\e}\subset N_{-E}$, hence
$$\Phi_E^{\e}({z})=\int_{N_{-E}MA{\check N}}\phi^{\e}(n_{z}g)\,dg\geq\int_{n_{z}^{-1}N_{\e}MA{\check N}}\phi^{\e}(n_{z}g)\,dg=\int_G \phi^{\e}(g)\,dg=1.$$
On the other hand, if ${z}\not\in E_{\e}^+$, then $n_{z}^{-1}N_{\e}\cap N_{-E}=\emptyset$, hence $\phi^{\e}(n_{z}g)=0$ for all $g\in N_{-E}MA{\check N}$ by uniqueness of the $NAM{\check N}$ decomposition, giving $\Phi_E^{\e}({z})=0$.
\end{proof}
We now relate the Burger-Roblin measure of $\phi^{\e}$ and Patterson-Sullivan measure of 
$\Phi^{\e}$
{.}
\begin{prop}\label{BRPS} There exist $C,c>0$ such that for all sufficiently small $\e>0$, we have
$$( 1-C\e)\,\om_{\psi} (
\Phi_{E^-_{c\e}}^\e
)\leq \int_{-E} {\tilde m}_\psi^{\BR}(\phi_{\vecz}^{\e})\,d\vecz\leq ( 1+C\e)\,\om_{\psi} (
\Phi_{E^+_{c\e}}^\e
).$$
\end{prop}
\begin{proof}
By Lemma \ref{PNdecomp},
we have
\begin{align*}
\int_{-E}& {\tilde m}_\psi^{\BR}(\phi_{\vecz}^{\e})\,d\vecz=\int_{-E}\int_{M}\int_N\int_{AM{\check N}}\phi^{\e}\big(n_{z'}am{\check n}\widetilde{m}n_{\vecz}\big)\\&\qquad\qquad \qquad\qquad\qquad\qquad \qquad\qquad \times e^{-\psi(\log a)}
{e^{\psi(\beta_{z'}(o,n_{z'} o))}}
\,d{\check n}\,d\widetilde{m}\,da\,d{\nu}_{\psi} ({z'})\,dm\,d\vecz
\\&= \int_{{\bb C^d}} \left(\int_{AM{\check N}N_{-E}}\phi^{\e}\big(n_{z'}am{\check n}n_{\vecz}\big) e^{-\psi(\log a)} d\vecz\,d{\check n}\,dm\,da\right) \,d\om_{\psi} ({z'}),
\end{align*}
where all the densities appearing in the expression are those of the corresponding Haar measures, except for {$d\nu_\psi$ and} $d\om_{\psi}$.
Note that if $\phi^{\e}(nam{\check n}n_{\vecz})\neq 0$, then $nam{\check n}n_{\vecz}\in N_{\e}A_{\e}M_{\e}{\check N}_{\e}$, hence
\begin{align}\label{eq.cp}
am{\check n} \in & AM{\check N}\, \cap\, \left( n^{-1}N_{\e}n_{\vecz}\exp\big(\Ad_{n_{-\vecz}}(\log(A_{\e}M_{\e}{\check N}_{\e}))\big)\right)\notag\\&\subset AM{\check N}\, \cap \, \big( n^{-1} n_{\vecz} N_{c'\e}A_{c'\e}M_{c'\e}{\check N}_{c'\e}\big)
\notag\\&= A_{c'\e}M_{c'\e}{\check N}_{c'\e} 
\end{align}
for some $c'\geq 1$ depending only on $E$. Decomposing the Haar measure $dg$ on $G$ according to $AM{\check N}N$ and then restricting to $A_{c'\e}M_{c'\e} {\check N}_{c'\e}N_{-E}$ gives
$$ e^{-\psi(\log a)} d\vecz\,d{\check n}\,dm\,da=\big(1+O(\e)\big)\,dg$$
since $a\in A_{c'\e}$ and $d g=d\vecz\,d{\check n}\,dm\,da$
for $g=am{\check n}n_{\vecz}$ \cite[Ch. 8]{Knapp}.
Hence
$$\int_{-E} m_\psi^{\BR}(\phi_{\vecz}^{\e})\,d\vecz=\big(1+O(\e)\big)\int_{{\bb C^d}} \int_{AM{\check N}N_{-E}}\phi^{\e}({n_{z'}}g)\,dg\,d\om_{\psi}({z'}), $$
with the implied constant depending only on $E$. Now using the maximum of $\|\Ad_{n_{\vecz}}\|$ over $\vecz\in \pm E$ together with the $NAM{\check N}$ decomposition of $\exp\big(\Ad_{n_{-\vecz}}(\log(A_{c'\e}M_{c'\e}{\check N}_{c'\e}))\big)$ as above gives  {the existence of} $c\geq c'$ such that  
$$ N_{-E^-_{c\e}}A_{\e}M_{\e} {\check N}_{\e}\subset  A_{c'\e}M_{c'\e} {\check N}_{c'\e}N_{-E}\subset  N_{-E^+_{c\e}}A_{c\e}M_{c\e} {\check N}_{c\e}.$$
Combined with \eqref{eq.cp}, {for every $z'\in \bb C^d$} we have
$$\int_{N_{-E^-_{c\e}}AM{\check N}}\phi^{\e}(
{n_{z'}}
g)\,dg\leq\int_{AM{\check N}N_{-E}}\phi^{\e}(
{n_{z'}}
g)\,dg
\leq \int_{N_{-E^+_{c\e}}AM{\check N}}\phi^{\e}(
{n_{z'}}
g)\,dg,$$
giving the desired inequality.
\end{proof}
Combining Lemma \ref{PhiE} and Proposition \ref{BRPS} gives {the following result.}
\begin{cor}\label{phibdd}
There exist $C,c>0$ such that for all $\e>0$ sufficiently small,
 $$( 1-C\e)\,\om_{\psi} \big(E^-_{(1+c)\e}\big)\leq \int_{-E} m_\psi^{\BR}(\phi_{\vecz}^{\e})\,d\vecz\leq ( 1+C\e)\,\om_{\psi} \big(E^+_{(1+c)\e}\big).$$
\end{cor}

\section{Equidistribution in average.}\label{sec.main}
We fix a positive $\Ga_\rho$-critical $\psi\in \frak a^*$, $\nu_\psi$ and $u=u_\psi$, continuing the notations from sections \ref{sec:count} and \ref{uSec}.
We also fix a closed cone
$\cal D\subset \op{int}\frak a^+$ such that
$\inte\cal D\supset \cal L_\rho-\{0\}$ and set 
  $D:=\exp \cal D$ as in \eqref{fixD} and \eqref{ddd}.
Recall the notation $B_\psi(E, R)=H\ba HKD_{\psi, R}
N_{-E}
$ for a bounded subset $E\subset \CC^d$, {and $\kappa$, $\ell>0$ given by Theorem \ref{prop.mixH}}.

The main goal of this section is to prove the following main technical ingredient of the proof of Theorem \ref{thm.main.intro}, using Proposition \ref{prop.mixH2}.
\begin{thm}\label{thm3} 
For any $f\in C_c(\GaG)$ and a bounded {measurable} subset $E\subset \CC^d$ 
such that $\om_{\psi}(\partial E)=0$,
\begin{align*}
 \lim_{R\rightarrow\infty} e^{- R} &\int_{
 B_{\psi}
 (E,R)} \int_{\Ga_\rho\cap H \ba H} f( {\Ga_\rho h} g)\,d[h]\,d[g]=c_{\Ga_\rho,\psi} \int_{-E}m^{\BR}_\psi (f_{\vecz})\,d\vecz, 
\end{align*}
where 
$ c_{\Ga_\rho,\psi}:=\tfrac{\kappa\, \op{sk}_{\Ga_\rho,\psi}(H)}{\pu}\left(\int_{\op{ker}\psi} e^{-{\ell} {I(\vecw)}}\,d\vecw\right)$
and $f_{\vecz}\in C_c(\Ga_\rho\ba G)$ is defined by
$f_{\vecz}(x):= \int_M f(x mn_{\vecz})\,dm.$
\end{thm}

In the above, $d[g]$ denotes
the $G$-invariant measure
on $H\ba G$ which is compatible to Haar measures
$dg$ and $dh$ on $G$ and $H$ respectively, that is, for any $f\in C_c(G)$,
$$ \int_G f(g)\,dg=\int_{H\ba G} \left( \int_H f(hg)\,dh\right)\,d[g].$$

\subsection*{Integral computation}
For each $\vecw\in\op{ker}\psi$, let $
Q_R(w)\subset (0, \infty)$ be defined as
\be\label{qrw}
Q_R(w):= \lbrace s\in \RR_{>0}\,:\, su+\sqrt s\vecw\in
A_{\psi, R}^+
\rbrace
{.}
\ee 
Since $\psi(w)=0$, we compute that for all $R>0$, $Q_R(w)$ is an interval of the form
$$Q_R(w)=(0, \sfrac{1}{\Phi_\rho(u)}R).$$
The uniform bound in Proposition \ref{prop.mixH2} enables us to
use the dominated convergence theorem to prove {the following result.}
\begin{lem}\label{lem.ab}
For $f\in C_c(\Ga_\rho\ba G)$, $\phi\in C_c(H)$ and a bounded {measurable subset} $E\subset\bb C^d$, define for each $\vecw\in\op{ker}\psi$,
$$
p_R(w)=e^{-R}\int_{E}\int_{Q_R(w)}s^{\frac{d-1}{2}}e^{2\sigma(su+\sqrt s w)}\int_H f_z({\Ga_\rho h}\exp(su+\sqrt sw))\phi(h)\,dh\,ds\,dz.
$$
Then 
\begin{enumerate}
    \item $\lim_{R\to \infty} p_R(w)= \sfrac{\kappa\, \mu_{H,\psi}^{\PS}(\phi)}{\pu }e^{-{\ell}I(\vecw)}\int_{E}m_\psi^{\BR}(f_z)\,dz$ and
    \item $p_R(w)\leq C e^{-\ell I(w)}$ for some $C=C(f,E,\phi)>0$.
\end{enumerate}
\end{lem}
\begin{proof} For simplicity, set $c_u=\Phi_\rho(u)$ in this proof.
For all sufficiently large $R>0$, we may rewrite $p_R(w)$ as
\begin{align*}
    p_R(w)&=e^{-R}\int_E\int_{0}^{R/{c_u}} e^{c_u s} J(s,w,z)\,ds\,dz\\
    &=\int_E\int_{-R/c_u}^{0}  e^{ c_u s{'}}J(s{'}+{R}/{c_u},w,z) \,ds{'}\,dz
\end{align*}
where 
$$
J(s,w,z)=s^{\frac{d-1}{2}}e^{(2\sigma-\psi)(su+\sqrt s w)}\int_H f_z({\Ga_\rho h}\exp(su+\sqrt sw))\phi(h)\,dh.
$$
By Proposition \ref{prop.mixH2}, $J(s,w,z)\to \kappa e^{-{\ell}I(w)} m_\psi^{\BR}(f_z)\mu_{H}^{\PS}(\phi)$ as $s\to\infty$ and 
$$J(s,w,z)\leq C'' e^{-\ell I(w)}$$ where $C''=C''(\sup_{z\in E}f_z,\phi)$ is as in Proposition \ref{prop.mixH2}.
Hence $(1)$ follows from the dominated convergence theorem {as $R\to\infty$}.
{Assertion} 
$(2)$ follows from the bound
$$
p_R(w)\leq \op{Vol}(E)\int_{-R/{c_u}}^{0}  e^{c_u s}J(s+R/{c_u},w,z) \,ds
$$
by setting
$C=
\frac{1}{c_u}
\op{Vol}(E)C''$.
\end{proof}

\subsection*{Proof of Theorem \ref{thm3}}
Without loss of generality, we may assume that $f\geq 0$. 
For $
[g]
\in H\ba G$, set
$$ f^H(
[g]
):=\int_{\Ga_\rho\cap H \ba H} f(
{
\Ga_\rho h
}
g)\,dh.$$
By  Proposition \ref{prop.UP},
{and using the expression of $g$ with respect to the generalized Cartan decomposition $G=HA^+K$,}
we can choose $\phi\in C_c({\Ga_\rho\cap H\ba}H)$ depending only on the support of $f$ and $E$ such that
\begin{equation}\label{eq.s}
f_{\vecz}^H(
[g]
)= \int_{{\Ga_\rho\cap H\ba}H} f_{\vecz}(
{
\Ga_\rho h
}
g)\phi(h)\,dh    
\end{equation} 
for all $\vecz\in E$.
This will allow us to apply Proposition \ref{prop.mixH2} directly to $f_{\vecz}^H$. Furthermore, by 
Proposition \ref{lem.supp}(1),
the support of $\muPS_{\Ga_\rho\cap H \ba H,\psi}$ is compact, so we may additionally assume that 
$\phi=1$ on the support of $\muPS_{\Ga_\rho\cap H \ba H, \psi}$ and
hence $\op{sk}_{\Ga_\rho,\psi}(H)=\muPS_{\Ga_\rho\cap H \ba H,\psi}(\phi)$.
By Proposition \ref{lem.supp}(2),
$$
\int_{H\ba H(A^+-D)KN_{-E}}f^H([g])\,d
 [g]<\infty.
$$
Since
$   d[a_t n_{\vecz}]=e^{2\sigma(\vect)}d\vect\,d\vecz$
where $\sigma(t)=\sum_{i=1}^{d} t_i$,
we deduce from
Lemma \ref{lem.XE} and the inclusion $B_\psi(E,R)\subset\tilde B_\psi(E,R)$ 
{that}
\begin{align}\label{eq.D}
& \limsup_{R\rightarrow\infty} e^{-R} \int_{{B_\psi}(E,R)}f^H(
[g])\,d [g]
\notag \\& \leq \limsup_{R\rightarrow\infty} e^{-R} \int_{-E^+_{
 \ell\e}} \int_{
 A_{\psi,R+\ell\e}^+
 } f^H(
 [
a_tn_{\vecz}]
)e^{2\sigma(\vect)}\,d\vect\,d\vecz.
\end{align}
Since $M\subset H$, we have 
\begin{align*}
f^H(
[
a_t n_{\vecz}
]
)=&f^H(
[m a_t n_{\vecz}])=f^H(
[a_t m n_{\vecz}]
)\\&=\int_M f^H(
[a_t m n_{\vecz}])\,dm= f_{\vecz}^H(
[a_t] ).
\end{align*}

We now compute the {upper} limit in \eqref{eq.D}.

Using {\eqref{qrw} and}  \eqref{eq.s} together with the fact that 
{
$t=su+\sqrt s w$ on $\frak a^+$ hence
} 
$dt=s^{\frac{d-1}{2}}\,ds\,dw$, we first rewrite 
\eqref{eq.D}
as
$
{\limsup_{R\to\infty}}
\int_{\op{ker}\psi}p_R(w)\,dw$
where
\begin{align*}
&p_R(w):=e^{-R}\int_{-E_{\ell\e}^+}\int_{Q_{R+
\ell\e}(\vecw)} s^{\frac{d-1}{2}}e^{2\sigma(su+\sqrt s w)} f_{\vecz}^H(
[\exp(su+\sqrt s\vecw)])\,ds\,dz\\
&=e^{-R}\int_{-E_{\ell\e}^+}\int_{Q_{R+
\ell\e}(\vecw)} s^{\frac{d-1}{2}}e^{2\sigma(su+\sqrt s w)} \int_{{\Ga_\rho\cap H\ba}H} f_{\vecz}(
{\Ga_\rho h}
\exp(su+\sqrt s\vecw)
)\phi(h)\,dh \,ds\,dz.
\end{align*}

Applying Lemma \ref{lem.ab} by replacing $R$ with $R+
\ell\e$ and $E$ with $-E_{\ell\e}^+$, by the dominated convergence theorem,
\begin{align*}
&\lim_{R\to\infty}\int_{\op{ker}\psi} p_R(w)\,dw=\int_{\op{ker}\psi} \lim_{R\to\infty} p_R(w)\,dw\\
&=\frac{\kappa\, \mu_{{\Ga_\rho\cap H\ba}H,\psi}^{\PS}(\phi) e^{
\ell\e}}{\pu} {\;}\int_{\op{ker}\psi}e^{-{\ell}I(\vecw)}\,dw\,\int_{-E_{
\ell\e}^+}m_\psi^{\BR}(f_{\vecz})\,dz.
\end{align*}
Altogether, we have thus obtained
$$ \limsup_{R\rightarrow\infty} e^{- R} \int_{{B}_\psi(E,R)}f^H(
[g])\,d
[g] \leq c_{{\Ga_\rho,}\psi}\, e^{{
\ell\e} }\int_{-E^+_{
\ell\e}} m_\psi^{\BR}(f_{\vecz})\,d\vecz.$$
Similarly,
but applying Lemma \ref{lem.XE} to $\tilde B_\psi(E_{\ell\e}^-,R-\ell\e)$ and $D_0=\exp \cal D_0$ where $\cal D_0$ is a cone such that $\cal L_\rho-\{0\}\subset\op{int}\cal D_0\subset \cl{\cal D_0}\subset \cal D$,
we have
$$ \liminf_{R\rightarrow\infty}   e^{- R} \int_{{B}_\psi(E,R)}f^H(
[g])\,d[g] \geq c_{{\Ga_\rho,}\psi}\, e^{{
-\ell\e}} \int_{-E_{
\ell\e}^-} m_\psi^{\BR}(f_{\vecz})\,d\vecz.$$
{
Note that by Corollary \ref{phibdd}, we have 
$$( 1-C\e)\,\om_{\psi} \big(E^-_{(1+c)\e}\big)\leq \int_{-E} m_\psi^{\BR}(\phi_{\vecz}^{\e})\,d\vecz\leq ( 1+C\e)\,\om_{\psi} \big(E^+_{(1+c)\e}\big)$$
for all sufficiently small $\e>0$.
}
Since $\om_{\psi} (\partial E)=0$, 
taking $\e\rightarrow0^+$ completes the proof.

\section{Proof of the main counting theorem}\label{pm}

In this section, we prove the following main theorem of this paper{.}
\begin{thm}\label{thm.main} Let $\cal P$ be a $\Ga_\rho$-admissible torus packing. 
For any positive linear form $\psi\in \frak a^*$,  there 
{exist}
a constant $c_\psi=c_{\cal P,\psi}>0$ 
 such that for any bounded {measurable} subset $E\subset \c^d$ with boundary contained in a proper real algebraic subvariety, we have
\begin{equation}\label{eq.E2}
\lim_{R\to\infty} e^{-\delta_\psi R}{\;} {N_R(\cal P,\psi,E)}=
c_{\psi}{\;}\omega_{\psi} (E).    
\end{equation}
\end{thm}

\begin{Ex} \label{tr} \rm 
Note that $\op{Vol}(T)=(2\pi)^d e^{-\sigma(\vecv(T))}$ since $\sigma(t_1,\cdots,t_d)=t_1+\cdots+t_d$.
Hence, we have $$
N_{R}(\cal P,\sigma,E)=\#\{T\in \cal P:
\text{Vol} (T) \ge (2\pi)^d e^{-R}, \; T\cap E\ne \emptyset\}.
$$ 

Since $\sigma\in \fa^*$ is positive, Theorem \ref{THM2} is a special case of Theorem \ref{thm.main}, with 
$
\delta_{L^1}(\rho)
=\delta_\sigma$, $c_{\cal P}=(2\pi)^{d\delta}c_{\cal P,\sigma}$ and $\om_{\psi}=\om_{\Ga_\rho,\sigma}$.
\end{Ex}

The proof of the following lemma is postponed until the final section (Theorem \ref{co1}){.}
\begin{lem}\label{lem.F}
For any bounded {measurable subset} 
 $E\subset\bb C^d$ with $\partial E$ contained in a proper real algebraic subvariety, we have $\om_\psi(\partial E)=0$.
\end{lem}

Since every homothety class of a positive linear form can be represented by a positive $\Ga_\rho$-critical linear form (Lemma \ref{lem.H}) and $\delta_\psi=1$ for critical linear forms,
Theorem \ref{thm.main} follows from
Lemma \ref{lem.F} and 
the following{.}

\begin{prop}\label{prop.C}
For any positive $\Ga_\rho$-critical linear form $\psi\in\frak a^*$
and any bounded {measurable subset} 
 $E\subset\bb C^d$ with $\om_\psi(\partial E)=\emptyset$,
we have
$$
\lim_{R\to\infty} e^{- R}{\;} {N_R(\cal P,\psi,E)}=
c_{\psi}{\;} \omega_{\psi} (E)    
$$ for some constant $c_\psi>0$.
\end{prop}

\subsection*{Special case: $\cal P=\Ga_\rho T_0$}
We will first prove Proposition \ref{prop.C} for the special case when $\cal P=\Ga_\rho T_0$.
This will allow us to apply the results obtained in previous sections.

Let $\cal D$ be as defined in \eqref{fixD}, and for any $R>0$, $A_R$ denote the $R$-neighborhood of $e$ in $A$. 
Fix closed cones $\cal D^\pm\subset\op{int}\frak a^+$ such that 
\begin{equation*}
    \cal L_\rho-\{0\}\subset \op{int}\cal D^-,\text{ }\cal D^--\{0\}\subset\op{int}\cal D,\text{ and }\cal D-\{0\}\subset \op{int}\cal D^+.
\end{equation*}
Let $D^\pm=\exp\cal D^\pm$ and $R_0>0$ be such that
$$
D^--A_{R_0}\subset\bigcap_{a\in A_1} Da,\qquad \bigcup_{a\in A_1} (D-A_{R_0})a\subset D^+.
$$
Recall the definitions of {$D_{\psi,R}^\pm$} and $E_\e^\pm$ from \eqref{eq.epm} 
{and \eqref{ddd}}.
Now defining 
\begin{align}\label{eq.bbb}
&{B}_\psi^0(E,R):=H\ba HK D_{[R_0,R)}N_{-E}\notag\\
&B_\psi^{\e}(E,R)^-:=H\ba HK  D_{[R_0,R)}^- N_{-E^-_{\e}}\text{ and }\notag\\
&B_\psi^{\e}(E,R)^+:=H\ba HK D_{[R_0,R)}^+ N_{-E^+_{\e}},
\end{align}
where $D^{\pm}_{[R_0,R)}=
{
D_{\psi,R}^\pm-A_{R_0}
}
$,
we have the following inclusions{.}

\begin{lem}\cite[Lemma 6.3]{OS} \label{KANstab}
For all $\e>0$ small enough, there exists a neighbourhood $\scrO_{\e}\subset G$ of the identity such that for all $R>R_0$,
$${B}_\psi^{\e}(E,R-\e)^-\subset  {B}_\psi^0(E,R)\scrO_{\e}\subset {B}_\psi^{\e}(E,R+\e)^+.$$
\end{lem}
We will now use the sets $B_\psi^{\e}(E,R\pm \e)^{\pm}$ to obtain the asymptotic of $\#\big([e]\Ga_\rho\cap {B}_\psi^0(E,R)\big)$.
Define functions $F_R$, and $F_R^{\e,\pm}$ on $\GaG$ by
$$ F_R([g]) := \sum_{\gamma\in (\Ga_\rho\cap H )\ba \Ga_\rho} \mathbbm{1}_{B_\psi^0(E,R)}(
{ H\ga g}
),$$
and 
\begin{equation}\label{Fpm} F_R^{\e,\pm}([g]) := \sum_{\gamma\in (\Ga_\rho\cap H )\ba \Ga_\rho} \mathbbm{1}_{B_\psi^{\e}(E,R\pm\e)^{\pm}}(
{ H\ga g}
).\end{equation}
Note that 
\begin{equation}\label{eq.FB}
F_R([e])=\#\big([e]\Ga_\rho\cap B_\psi^0(E,R)\big),    
\end{equation}
and by Lemma \ref{KANstab}, we have
$$ F_R^{\e,-}([ g])\leq F_R([e]) \leq F_R^{\e,+}([ g])$$
for all $[g]\in [e] \scrO_{\e}$ and all $\e$ {small enough and} less than the injectivity radius of $[e]\in\Ga_\rho\ba G$.
Now fix  any non-negative function $\phi^{\e}\in C_c(
[e]
\scrO_{\e})$  such that $\int \phi^{\e}([g])\,d[g]=1$
where $d[g]$ is a Haar measure on $\Ga_\rho\ba G$.
Then 
\begin{equation}\label{eq.WL}
    \langle F_R^{\e,-},\phi^{\e}\rangle\leq F_R([e]) \leq \langle F_R^{\e,+},\phi^{\e}\rangle. 
\end{equation}
where $\langle \psi_1, \psi_2\rangle=\int_{\Ga_\rho\ba G} \psi_1(
[g]
)\psi_2(
[g]
) d[g]$
whenever the integral converges.
We will use Theorem \ref{thm3} to estimate the integrals $ \langle F_R^{\e,\pm},\phi^{\e}\rangle$ (cf. \cite[(6.6), p. 30]{OS} and \cite[Proposition 9.10]{ELO}){.} 
\begin{prop}\label{equiint}
For any $\e>0$ small enough, {we have}
$$\langle F_R^{\e,\pm},\phi^{\e}{\rangle} \sim c_{\psi} e^{R\pm \e}\int_{-E_{\e}^{\pm}} \mBR(\phi_{\vecz}^{\e})\,d\vecz \quad\text{as $R\rightarrow\infty$},$$
where the constant $c_\psi =c_{\Ga_\rho,\psi}$ is given in Theorem \ref{thm3}.
\end{prop}
\begin{proof}
Using unfolding, we have
\begin{align*}
\langle F_R^{\e,\pm},\phi^{\e}\rangle &=\int_{\GaG}\left( \sum_{\gamma\in (\Ga_\rho\cap H )\ba \Ga_\rho} \mathbbm{1}_{B_\psi^{\e}(E,R\pm\e)^{\pm}}(
{ H\ga g}) \right)\phi^{\e}([g])\,dg
\\& = \int_{\Ga_\rho\cap H \ba G} \mathbbm{1}_{B_\psi^{\e}(E,R\pm\e)^{\pm}}({Hg}) \phi^{\e}([g])\,dg
\\& =\int_{B_\psi^{\e}(E,R\pm\e)^{\pm}}\left(\int_{\Ga_\rho\cap H \ba H} \phi^{\e}({\Ga_\rho hg})dh \right)\,{d(Hg)}.
\end{align*}
Since the set difference between $B_\psi^\e(E,R\pm\e)^\pm$ and $B_{D^\pm,\psi}(E_\e^\pm,R\pm\e)$ is bounded independent of $R$,
Theorem \ref{thm3} then gives the claimed identity. 
\end{proof}

\begin{proof}[Proof of Proposition \ref{prop.C} when $\cal P=\Ga_\rho T_0$]
Note that the set difference between $B_\psi^0(E,R)$ and $B_{D,\psi}(E,R)$ is bounded independent of $R$.
Hence,
by Proposition \ref{BETcount1} and
Corollary \ref{cor.C},
\begin{align*}
 &\liminf_{R\to\infty} e^{-R}\#([e]\Ga_\rho \cap  
 B_\psi^0(E_{\e}^-,R)
 ) \leq\liminf_{R\to\infty}  e^{-R} N_R(\cal P,\psi,E)\\
&\leq\limsup_{R\to\infty}  e^{-R} N_R(\cal P,\psi,E)\leq \limsup_{R\to\infty}  e^{-R} \#([e]\Ga_\rho \cap 
B_\psi^0(E_{{{(\ell'+1)}}\e}^+,R+\ell{'}\e)
).
\end{align*}
Let $\e_0=(\ell{'}+2)\e$.
The above computation, combined with \eqref{eq.FB}, \eqref{eq.WL} and Proposition \ref{equiint} gives
\begin{align*} 
&c_{\psi} e^{-
\e_0
}\int_{-E_{
\e_0
}^{-}} \mBR(\phi_{\vecz}^{
\e_0
})\,d\vecz\leq \liminf_{R\rightarrow\infty} e^{-R} N_{R}(\cal P,\psi, E)\\&\leq \limsup_{R\rightarrow\infty} e^{-R }{N_{R}(\cal P, \psi, E)}\leq c_{\psi} e^{
\e_0
}\int_{-E_{
\e_0
}^{+}} \mBR(\phi_{\vecz}^{
\e_0
})\,d\vecz.
\end{align*}
Corollary \ref{phibdd} now gives
\begin{align*} 
& c_{\psi} e^{-
\e_0
}( 1-C
\e_0
)\,\om_{\psi} \big(E^-_{(2+c)
\e_0
}\big)\leq \liminf_{R\rightarrow\infty} e^{-R} N_{R}(\cal P,\psi, E)\\&\leq \limsup_{R\rightarrow\infty} e^{-R} N_{R}(\cal P,\psi, E)\leq c_{\psi} e^{
\e_0
}( 1+C
\e_0
)\,\om_{\psi} \big(E^+_{(
2+c)
\e_0
}\big).
\end{align*}
Since $\om_{\psi} (\partial E)=0$ by Lemma \ref{lem.F}, the
regularity of $\om_{\psi} $ gives
$$ \lim_{\e\rightarrow 0^+} c_{\psi} e^{\pm\e}( 1\pm C\e)\,\om_{\psi} \big(E^{\pm}_{(2+c)\e}\big)=c_{\psi} \,\om_{\psi} (E),$$
completing the proof.
\end{proof}

\subsection*{General case} Without loss of generality, we may assume that $\scrP$ consists of a single $\Ga_\rho$-orbit; hence
let $\scrP=\Ga_\rho T$ be a $\Ga_\rho$-admissible torus packing.
We write
$$ T= g_0T_0,$$
where $g_0=n_{z_0}{a_{t_0}}$; here $z_0$ is the vector consisting of the centers of the circles of $T$ and {$t_0=\log r_0$ where} $r_0=(r_1,\ldots,r_d)$ are the corresponding radii. 
Set 
$$
\Ga_\rho^{g_0}:=g_0^{-1}\Ga_\rho g_0.
$$

Note that
\begin{align*}
&N_R(\scrP,\psi,E)=\#\left\lbrace T'\in\Ga_\rho T\,:\, T'\cap E\neq \emptyset\;\mathrm{and}\;\psi(\vecv(T'))\leq R\right\rbrace    
\\&=\#\left\lbrace T'\in\Ga_\rho g_0T_0\,:\, T'\cap E\neq \emptyset\;\mathrm{and}\;\psi(\vecv (T'))\leq R\right\rbrace
\\&=\#\left\lbrace T'\in\Ga_\rho^{g_0} T_0\,:\, g_0T'\cap E\neq \emptyset\;\mathrm{and}\;\psi(\vecv (g_0T'))\leq R\right\rbrace.
\\&=\#\left\lbrace \gamma\in(\Ga_\rho^{g_0}\cap H)\ba\Ga_\rho^{g_0}\,:\, g_0\ga^{-1}T_0\cap E\neq \emptyset\;\mathrm{and}\;\psi(\vecv (g_0\ga^{-1}T_0))\leq R\right\rbrace.
\end{align*}
Similarly to Proposition \ref{BETcount1}, we can 
obtain the following estimate of
$N_R(\scrP,\psi,E)$ in terms of $
\tilde B_\psi(E_{\e}^\pm,R)
${.} 
\begin{prop} \label{BETg}
For any $\e>0$, there exists $q_0=q_0(\scrP,\e)>0$ such that for any $R>0$ and {any} bounded {measurable} 
 subset $E\subset \CC^d$, we have
$$
\#\big([e]\Ga_\rho^{g_0} \cap
\tilde B_\psi(E_\e^-,R)
g_0\big)-q_0 \leq N_{R}(\cal P,\psi,E)\leq \#\big([e]\Ga_\rho^{g_0}\cap
\tilde B_\psi(E_\e^+,R)
g_0\big)+q_0.
$$
\end{prop}

Note that $\Ga_\rho^{g_0}$ is also a self-joing of convex cocompact representations.
Let $\La_{\rho}^{g_0}$ and $\cal L_{\rho}^{g_0}$ denote its limit set and limit cone, respectively.
It is immediate from the definition that
$$
\La_{\rho}^{g_0}=g_0^{-1}\La_{\rho}\text{ and }\cal L_{\rho}^{g_0}=\cal L_{\rho}.
$$
Now, writing $g_0=(g_{0,1},\cdots,g_{0,d})$, the homeomorphisms in \eqref{eq.EQV} associated to $\Ga_\rho^{g_0}$ can be written as $g_{0,i}^{-1}f_i g_{0,1}$ $(1\leq i\leq d)$.
A direct computation shows that $T$ is $\Ga_\rho$-admissible if and only if $T_0$ is $\Ga_\rho^{g_0}$-admissible.
Hence, we can apply the results obtained in previous sections for a new subgroup $\Ga_\rho^{g_0}$.

\subsection*{Transition from $\Ga_\rho$ to $\Ga_\rho^{g_0}$}
Let $\Phi_\rho^{g_0}=\Phi_{\Ga_\rho^{g_0}}$ denote the growth indicator function associated to $\Ga_{\rho}^{g_0}$.
The following lemma is standard and can be proved {using \cite[Lem. 4.6]{B}, \cite[Lem. 3.1.6]{Qu1} and} the definition of $\Phi_\rho${.} 
\begin{lem}\label{lem.G}
We have
$$ \Phi_{\rho}^{g_0}=\Phi_{\rho}.$$
\end{lem}

Since $\psi$ is $\Ga_\rho$-critical, it follows from Lemma \ref{lem.G} that $\psi$ is $\Ga_\rho^{g_0}$-critical.
{The unique unit vectors, as provided by Lemma \ref{lem.1-1} remain the same, regardless of whether we view $\psi$ as a $\Ga_\rho$-critical linear form or $\Ga_\rho^{g_0}$-critical linear form.}
Let $\nu_{\psi}^{g_0}$ denote the
$(\Ga_\rho^{g_0},\psi)$-$\PS$ {probability} measure supported on $\La_{\rho}^{g_0}$.
Define a measure
 $\widetilde{\nu}_{\psi}^{g_0}$
 on $\hat{\CC}^d$ via the formula
 $$ d\widetilde{\nu}_{\psi}^{g_0}(z)=e^{-\psi( \beta_{z}(o,g_0\cdot o))}d((g_0)_*\nu_{\psi}^{g_0})(z).$$
 \begin{lem}\label{lem.n}
 We have:
 $$
 \frac{\widetilde{\nu}_{\psi}^{g_0}}{|\widetilde{\nu}_{\psi}^{g_0}|}=\nu_{\psi}.
 $$
 \end{lem}
 \begin{proof}
  Since the support of $\nu_{\psi}^{g_0}$ is $\La_\rho^{g_0}=g_0^{-1}\La_\rho$, we have
  $$
  (g_0)_*\nu_{\psi}^{g_0}({\Lambda_{\rho}})=\nu_{\psi}^{g_0}(g_0^{-1}\Lambda_\rho)=1.
  $$
  Therefore, $\widetilde{\nu}_{\psi}^{g_0}$ is also supported on $\Lambda_\rho$.
  Furthermore, for any $\ga\in\Ga_\rho$, we have
\begin{align*}d\ga_*\widetilde{\nu}_{\psi}^{g_0}(z)&=e^{-\psi( \beta_{\ga^{-1}{\;} z}(o,g_0{\;} o))}d\ga_*(g_0)_*\nu_{\psi}^{g_0}(z)
\\&=e^{-\psi( \beta_{ z}(\ga{\;} o,\ga g_0{\;} o))}d(g_0)_{*}(g_0^{-1} \ga g_0)_*\nu_{\psi}^{g_0}(z)
\\&=e^{-\psi(\beta_{ z}(\ga{\;} o,\ga g_0{\;} o))}e^{-\psi(\beta_{g_0^{-1}{\;} z}(g_0^{-1}\ga g_0{\;} o,o))}d (  g_0)_*\nu_{\psi}^{g_0}(z)
\\&=e^{\psi(\beta_{z}(\ga g_0{\;} o,\ga{\;} o))}e^{\psi(\beta_{ z}(g_0{\;} o,\ga g_0{\;} o))}e^{\psi( \beta_{z}(o,g_0{\;} o))} d\widetilde{\nu}_{\psi}^{g_0}(z)
\\&=e^{\psi(\beta_{z}(o,g_0{\;} o))} e^{\psi(\beta_{ z}(g_0{\;} o,\ga g_0{\;} o))}e^{\psi(\beta_{ z}(\ga g_0{\;} o,\ga{\;} o))}d\widetilde{\nu}_{\psi}^{g_0}(z)
\\&=e^{\psi(\beta_{z}(o,\ga{\;} o))}d\widetilde{\nu}_{\psi}^{g_0}(z),
\end{align*}
i.e. 
$$ \frac{d\ga_*\widetilde{\nu}_{\psi}^{g_0}}{d\widetilde{\nu}_{\psi}^{g_0}}(z)=e^{-\psi( \beta_{z}(\ga{\;} o,o))}.$$
 This shows that $\frac{\widetilde{\nu}_{\psi}^{g_0}}{|\widetilde{\nu}_{\psi}^{g_0}|}$ is a $(\Ga_\rho,\psi)$-$\PS$ {probability} measure.
 By \cite[Thm. 1.3]{LO}, $\nu_{\psi}$ is the unique 
 $(\Ga_\rho,\psi)$-$\PS$ {probability} measure,
 and hence the lemma is proved.
 \end{proof}
 Let $\om_{{g_0},\psi}$ denote the measure defined as in \eqref{eq.om}, associated to $\Ga_\rho^{g_0}$ and $\psi$.
 Using Lemma \ref{lem.n}, we can now show 
 {the following result.}
 \begin{lem}\label{omegg0}
 There exists $c_{g_0}>0$ such that
  $$\omega_{g_0, \psi}(g_0^{-1}{\;} E)=c_{g_0}{\;}\omega_{\psi}(E) $$
  for all Borel sets $E\subset \CC^d$.
 \end{lem}
 \begin{proof}
By { Definition \ref{eq.om}}, 
 \begin{align*}
\omega_{g_0, \psi}(g_0^{-1}{\;} E)=\int_E d (g_0)_*\omega_{g_0, \psi}(z)= \int_E e^{\psi( \beta_{g_0^{-1}{\;} z}(o,n_{g_0^{-1}z}o))}\,d(g_0)_*\nu_{\psi}^{g_0}(z).
 \end{align*}
 Now writing {as seen above} $g_0=n_{z_0}a_{t_0}$, we thus have {by Lemma \ref{lem.n}}
 \begin{align*}
\omega_{g_0, \psi}(g_0^{-1}{\;} E)=& \int_E e^{\psi(\beta_{g_0^{-1}{\;} z}(o,g_0^{-1}n_{z}a_{t_0}{\;} o))} e^{\psi(\beta_{z}(o,g_0{\;} o))}\,d\widetilde{\nu}_{\psi}^{g_0}(z)
 \\&=|\widetilde{\nu}_{\psi}^{g_0}| \int_E e^{\psi(\beta_{ z}(g_0{\;} o,n_{z}a_{t_0}{\;} o))}e^{\psi(\beta_{z}(o,g_0{\;} o))}\,d\nu_{\psi}^{}(z)
  \\&=|\widetilde{\nu}_{\psi}^{g_0}| \int_E e^{\psi(\beta_{ z}( o,n_{z}a_{t_0}{\;} o))}e^{-\psi(\beta_{z}(o,n_z{\;} o))}\,d\omega_{\psi}(z)
  \\&=|\widetilde{\nu}_{\psi}^{g_0}| \int_E e^{\psi(\beta_{ z}( n_z{\;} o,n_{z}a_{t_0}{\;} o))}\,d\omega_{\psi}(z)\\
  &=|\widetilde{\nu}_{\psi}^{g_0}| e^{\psi( \beta_{ 0}(  o, a_{t_0}{\;} o))}\omega_{\psi}(E)
  \\&= |\widetilde{\nu}_{\psi}^{g_0}|\, e^{-\psi( t_0)}\omega_{\psi}(E),
 \end{align*}
 as desired.
 \end{proof}
Next, let $m_{\Ga_\rho^{g_0},\psi}^{\BR}=m^{\BR}_{g_0, \psi}$ denote the Burger-Roblin measure associated to $\Ga_\rho^{g_0}$ and the linear form $\psi$. 
{
Denote the right $G$-action on functions on $\Ga_\rho^{g_0} \ba G$ by $(g\cdot f)([ h])=f([ h g])$ and let $G_\e$ be the $\e$-neighborhood of $e$ in $G$.
For any $\phi^{\e}\in C_c(\Ga_\rho^{g_0}\ba G)$ whose support is contained in $[e]G_\e$, we have the following.
}
\begin{lem}\label{BRconj1} For all small enough $\epsilon>0$, 
$$\int_{-E_{\e}^{\pm}} m^{\BR}_{g_0, \psi}\big((g_0\cdot\phi^{\e})_z\big)\,dz=e^{
\psi({t_0})} \int_{-g_0^{-1}{\;} E_{\e}^{\pm}} m^{\BR}_{g_0, \psi} \big(\phi^{\e}_z\big)\,dz.$$
\end{lem}
 \begin{proof}
Denote the 
$\Ga_\rho^{g_0} $
-invariant lift of $m^{\BR}_{g_0, \psi}$ to $G$ by $\widetilde{m}^{\op{BR}}_{g_0,\psi}$.
We will use the $G=KA{\check N}$ decomposition to write
\begin{equation}\label{KAN} d\widetilde{m}^{\op{BR}}_{g_0,\psi}(kau)=e^{-
\psi({\log a})
} du\,da\,d\cl{\nu}_\psi^{g_0}(k) \end{equation}
{by \cite[Lem. 4.9]{ELO},}
 where the measure $\cl{\nu}_\psi^{g_0}$ on $K$ is defined by 
 $$\int_{K}f(k)\,d\cl{\nu}_\psi^{g_0}(k)=\int_{K/M}\int_M f(km)\,dm\,d\nu_{\psi}^{g_0}(k^+)$$ for all $f\in C(K)$.
 For any $f\in C_c(G)$, and {measurable} $L\subset \CC^d$, using the fact that $g_0=n_{z_0}a_{{t_0}}$, we have
 \begin{align*}
\int_{-L} \widetilde{m}^{\op{BR}}_{g_0,\psi}\big((g_0\cdot f)_z\big)\,dz& =\int_{-L}\int_M \int_G f(gmn_zg_0)\,d\widetilde{m}^{\op{BR}}_{g_0,\psi}(g)\,dm\,dz\\
 =&\int_{z_0-L}\int_M \int_G f\big(gmn_za_{{t_0}}\big)\,d\widetilde{m}^{\op{BR}}_{g_0,\psi}(g)\,dm\,dz\\
 =e^{2\sigma({t_0})}&\int_{a_{-{t_0}}{\;}(z_0-L)}\int_M \int_G f\big(ga_{{t_0}}mn_z\big)\,d\widetilde{m}^{\op{BR}}_{g_0,\psi}(g)\,dm\,dz.
 \end{align*}
 From \eqref{KAN}, we obtain
 \begin{align*}\int_G &f\big(ga_{{t_0}}mn_z\big) \,d\widetilde{m}^{\op{BR}}_{g_0,\psi}(g)=\int_{KA{\check N}}f\big(kau a_{{t_0}}mn_z\big) e^{-\psi(\log a)} du\,da\,d\cl{\nu}_\psi^{g_0}(k)\\
 &=e^{-2\sigma({t_0})}\int_{KA{\check N}}f\big(kaa_{{t_0}}u mn_z\big) e^{-\psi (\log a)} du\,da\,d\cl{\nu}_\psi^{g_0}(k)
 \\
&=e^{(\psi-2\sigma)({t_0})}\int_{G}f\big(g mn_z\big) \,d\widetilde{m}^{\op{BR}}_{g_0,\psi}(g).
 \end{align*}
This gives
$$\int_{-L} \widetilde{m}^{\op{BR}}_{g_0,\psi}\big((g_0\cdot f)_z\big)\,dz=e^{\psi({t_0})}\int_{-g_0^{-1}{\;} L} \widetilde{m}^{\op{BR}}_{g_0,\psi}\big( f_z\big)\,dz,   $$ 
proving the claim.
 \end{proof}

 \begin{proof}[Proof of Proposition \ref{prop.C} for the general case]
We define a counting function {that we again denote by} $F_R$, on $\Ga_\rho^{g_0}\ba G$ by 
$$ F_R([g])=\sum_{\ga\in(\Ga_\rho^{g_0} \cap H)\ba\Ga_\rho^{g_0} }\mathbbm{1}_{
{B}^0_{\psi}
(E,R)}(\ga g);$$
We have:
\begin{equation}\label{eq.e1}
    \#\left( [e]\Gamma^{g_0} \cap 
{B}_\psi^0
    (E,R)g_0 \right) =(g_0^{-1}\cdot F_R)([e]). 
\end{equation}
Let $B_\psi^{\e}(E,R)^{\pm}$ be as in \eqref{eq.bbb} and $\cal O_\e$ be as in Lemma \ref{KANstab}.
Set {now}
\begin{equation}\label{Fpm2} F_R^{\e,\pm}([g]) := \sum_{\gamma\in (\Ga_\rho^{g_0} \cap H)\ba \Ga_\rho^{g_0} } \mathbbm{1}_{ {B}_\psi^{\e}(E,R\pm\e)^{\pm}}(\ga g).\end{equation}
Then
$$ (g_0^{-1}\cdot F_R^{\e,-})([g])\leq (g_0^{-1}\cdot F_R)([e]) \leq (g_0^{-1}\cdot F_R^{\e,+})([g])$$
for all $g\in\scrO_{\e}$.
Thus, choosing a non-negative $\phi^{\e}\in C_c({\Ga_\rho^{g_0}\ba G})$ {with support in $[e]\cal O_\e$} such that $
{
\int
}
\phi^{\e}([g])\,d
{[g]}
=1$ gives
\begin{equation}\label{eq.e2}
     \langle F_R^{\e,-},g_0\cdot\phi^{\e}
     { \rangle}
     \leq F_R([g_0]
) \leq \langle F_R^{\e,+},g_0\cdot\phi^{\e}
     { \rangle}.
\end{equation}
Similarly as in Proposition \ref{equiint}, we have
\begin{equation}\label{equiint2}\langle F_R^{\e,\pm},g_0\cdot\phi^{\e}
     { \rangle}
     \sim c{\;}  e^{R\pm \e}\int_{-E_{\e}^{\pm}} {m}^{\op{BR}}_{g_0,\psi}\big((g_0\cdot\phi^{\e})_z\big)\,dz, 
\end{equation}
as $R\rightarrow\infty$, where $c=c_{\Ga_\rho^{g_0},\psi}$.

Similarly to
 the proof for the special case,
 from {Corollary \ref{cor.C},}
 Proposition \ref{BETg}, 
 \eqref{eq.e1}, \eqref{eq.e2} and \eqref{equiint2},
 we obtain {for $\e_0=(\ell'+2)\e$,}
\begin{align*}
&c{\;}  e^{-{\e_0}}\int_{-E_{{\e_0}}^{-}} m^{\BR}_{g_0,\psi }\big((g_0\cdot\phi^{{\e_0}})_z\big)\,dz \leq\liminf_{R\rightarrow\infty}e^{-R} N_R(\scrP,{\psi,}E)\\&\leq \limsup_{R\rightarrow\infty} e^{-R}N_R(\scrP,{\psi,}E)\leq c {\;} e^{-{\e_0}}\int_{-E_{{\e_0}}^{+}} m^{\BR}_{g_0,\psi}\big((g_0\cdot\phi^{{\e_0}})_z\big)\,dz.
\end{align*}

Applying Lemma \ref{BRconj1} and Corollary \ref{phibdd} gives
\begin{align*}e^{\psi({t_0})}&( 1-C{\e_0})\,\omega_{g_0,\psi}\big(g_0^{-1}{\;} E^-_{(
2
+c){\e_0}}\big)\leq \int_{
-E_{{\e_0}}^{\pm}
} {m}^{\op{BR}}_{g_0,\psi}\big((g_0\cdot\phi^{{\e_0}})_z\big)\,dz
\\&\leq e^{\psi({t_0})}(1+C{\e_0})\,\omega_{g_0,\psi}\big(g_0^{-1}{\;} E^+_{(
2
+c){\e_0}}\big). \end{align*}
We now use Lemma \ref{omegg0} to change 
$\omega_{g_0,\psi}$
to $\omega_{\psi}$ and then taking $\epsilon\rightarrow 0$ as in the case $T=T_0$ gives
$$ \lim_{R\rightarrow \infty} e^{-R} N_R(\scrP,\psi,E) =c_0 {\;} \omega_{\psi}(E)$$
for some positive constant $c_0>0$. 
Since the left-hand side of the above equation does not depend on the choice of $g_0$, we in fact have that $c_0$ cannot depend on $g_0$ either, proving the theorem.
 \end{proof}

\subsection*{On Remark \ref{Euc}(2)}\label{Euc2}
There exists a unique vector $u=u_{\Ga_\rho}\in { (\br_{\ge 0})}^d$ such that
$\Phi_\rho(u)=\max\{\Phi_\rho(v)\;:\; \|v\|\le 1\}$, called the direction of the maximal growth of $\Ga_\rho$. Moreover,  $u\in \inte \L_\rho$ (\cite{Sa}, \cite{PS}). 
Let $\psi=\psi_u$ be as defined in Lemma \ref{lem.1-1}. Then
 for all $w\in \ker \psi$, the subset
\be\label{qrw2}
Q_R(w):= \lbrace s\in \RR_{>0}\,:\, \|su+\sqrt s\vecw\| < R \rbrace \ee 
is an interval of the form
$(0,\sfrac{1}{2}\big(-\|\vecw\|^2+\sqrt{\|\vecw\|^4+4R^2}\big))$.
Then using \cite[Lem. 9.4]{ELO} substituting Lemma \ref{lem.ab},
our proof yields the following:
\begin{equation}\label{eq.E3}
 \lim_{R\to \infty} e^{-\delta_{\Ga_\rho}{R}} {\#\{T\in \cal P: \|\vecv(T)\|  <R, \, T\cap E\ne \emptyset\} } =
c_{\psi}{'}{\;} \omega_\psi  (E\cap \Lambda_\rho)
\end{equation}
where $\delta_{\Ga_\rho}=\Phi_\rho (u)$ and $c_\psi{'}>0$.

We remark that whereas Lemma \ref{lem.ab} relied on the uniformity in the mixing Theorem \ref{m11}, \cite[Lem. 9.4]{ELO} did not need such uniformity.
The reason is that for each $w\in\op{ker}\psi$, the amount of time the trajectory $s\mapsto su+\sqrt sw$ spends in the set $\{v\in\frak a^+:\norm{v}<R\}$ is much less than the time it spends in $\{v\in\frak a^+:\psi(v)<R\}$ to the extent that when viewed from a proper scale, it gives rise to an $L^1$-function on $\op{ker}\psi$.

 \section{PS-measures are null on algebraic varieties}\label{sec:f}
In this section, we prove that $\nu_\psi(S)=0$ for any proper real algebraic subvariety $S$ of $\hat{\bb C}^d$.
Since $\om_{\psi}$ is absolutely continuous with respect to $\nu_\psi$, it follows that $
\om_\psi
(\partial E)=0$ whenever $E$ has boundary contained in a proper real algebraic subvariety, and in particular, Lemma \ref{lem.F} follows.

We will in fact prove this in a more general setup, which we now explain.

Let $G$ be any connected semisimple linear real algebraic group. Let $P=MAN<G$ be a minimal parabolic subgroup with a fixed Langlands decomposition. Let $\fa$ denote the Lie algebra of $A$.
Let $\i$ denote the opposition involution on $\fa$.
We remark that 
the opposition involution is non-trivial if and only if
$G$ has a simple factor of type $A_n$ ($n\ge 2$), $D_{2n+1}$ ($n\ge 2$) and $E_6$ \cite[1.5.1]{Ti}. 
For instance, when $G$ is a product of rank one groups, $\i$ is trivial.

A Borel probability measure $\nu$ on $\cal F=G/P$ is called a $(\Gamma,\psi)
$-conformal measure for a linear form $\psi\in \fa^*$ if 
for all $\ga\in \Ga$ and $\xi\in \cal F$,
$$\frac{d\ga_*\nu}{d\nu}(\xi)= e^{-\psi(\beta_{\xi}(\ga , e))},$$
where $\beta$ denotes the $\fa$-valued Busemann map \cite[Def. 2.3]{ELO}. When supported on the limit set $\La$, it is called a $(\Ga, \psi)$-PS-measure.

We recall the following result:
consider the diagonal action of $\G$ on $\cal F\times\cal F$.

\begin{prop} \cite[Prop. 6.3]{LO2} \label{prop.Z}
Let $\G<G$ be a \Zar Anosov subgroup of $G$ (with respect to $P$). Let $\psi\in \fa^*$ be a  linear form.
Let $\nu$ and $\nu_{\i}$ be  respectively
 $(\Ga,\psi)$ and $(\Ga, \psi\circ \i)$-PS measures. Then  $(\F\times \F,\nu\times \nu_{\i}, )$ is $\Ga$-ergodic. 
\end{prop}

\begin{thm}\label{co1} Let $\G<G$ be a \Zar Anosov subgroup of $G$.
 For any $(\Ga,\psi)$-PS measure $\nu$ for some $\psi\in \fa^*$ with $\psi\circ \i =\psi $,
we have $$\nu(S)=0$$ 
for any proper real algebraic subvariety $S$ of $\cal F$.
\end{thm}

Theorem \ref{co1} follows from the following  by Proposition \ref{prop.Z}.
\begin{thm}\label{mg}
 Let $\G<G$ be a discrete subgroup and  $\nu$ be a
 $(\Ga,\psi)$-PS measure for some $\psi\in \fa^*$ such that the diagonal $\Ga$-action on $(\F\times \F, \nu\times \nu )$  is ergodic.
 Then  $$\nu(S)=0$$ 
for any proper real algebraic subvariety $S$ of $\cal F$.
\end{thm}
\begin{proof}
Suppose the theorem is false.
Let $S$ be a proper subvariety of $\F$ with $\nu(S)>0$
and of minimal dimension. 
We may assume without loss of generality that $S$ is irreducible.

Since $(\nu\times \nu) (S\times  S)=\nu(S) \times \nu( S)>0$  the $\G$-ergodicity of $\nu\times \nu$ implies that
 $\Ga(S\times S)$ must have full $\nu\times \nu$-measure.
 Since for any $\ga_0\in \Ga$,
 $(\nu\times \nu)(S\times \ga_0 S)>0$,
 there must exist $\ga\in \Ga$ such that
 $(S\cap \ga_0S)\cap (\ga S\times \ga S)$ has positive $\nu\times \nu$-measure. This implies that $\nu(S\cap \ga S)>0$ and
$\nu(\ga_0 S\cap \ga S)>0$. Since $S$ is an irreducible variety, for any $\ga\in \G$,
either $S=\ga S$ or
the dimension of $S\cap \ga S$ is strictly smaller than that of $S$, and hence $\nu(S\cap \ga S)=0$. Therefore $S=\ga S=\ga_0 S$.
Since $\ga_0$ was arbitrary, it follows that $\Gamma S=S$; a contradiction to the { Zariski-density of $\Ga$}.
\end{proof}

We deduce the following corollary when $G$ is of rank one. In this case,
$G=\op{Isom}^+(X)$ for a rank one symmetric space $X$ and $\F$ is equal to the geometric boundary of $X$. For a non-elementary discrete subgroup $\Ga<G$ of divergence type
(e.g., geometrically finite),
there exists a unique $\Ga$-conformal measure, say $\nu_\Ga$,
 of dimension equal to the critical exponent
$\delta_\Ga$ and the diagonal $\Ga$-action on $(\F\times \F, \nu_\Ga\times \nu_\Ga)$ is ergodic
\cite[Thm. 1.7]{Rob}. Therefore we obtain
 \begin{cor} Let $G$ be of rank one  and $\G<G$ be a  \Zar discrete
 subgroup of divergence type. Then $\nu_\Ga(S)=0$ for any proper real algebraic subvariety $S$ of $\F$.
     
 \end{cor}
 This corollary was obtained in \cite{FSp} when $G=\op{SO}(n,1)$ and $\Ga<G$ is geometrically finite.

\end{document}